\documentclass{amsart}
\usepackage{amsmath,amssymb,amsthm,amscd}
\usepackage[a4paper,text={140mm,210mm},centering,headsep=5mm,footskip=10mm]{geometry} 
\usepackage[matrix,arrow, curve]{xy}
\numberwithin{equation}{section}
\usepackage{tikz}
\usepackage{color}
\usepackage{graphicx}
\usepackage[hypertexnames=false,linkcolor=black, colorlinks=true, citecolor=black, filecolor=black,urlcolor=black]{hyperref}

\definecolor{mediumgray}{HTML}{888888}
\definecolor{dimgray}{HTML}{BBBBBB}
\definecolor{brightred}{HTML}{FF5555}
\definecolor{codehighlightcolor}{HTML}{AF5FD7}

\definecolor{darkgreen}{rgb}{0.0, 0.7, 0.0}
 \definecolor{bananayellow}{rgb}{1.0, 0.88, 0.21}
\definecolor{choc}{rgb}{0.82, 0.41, 0.12}
 	
\definecolor{softgray}{gray}{0.55}
\definecolor{verysoftgray}{gray}{0.75}

\newcounter{CountAlpha}

\newtheorem{ThmIntro}[CountAlpha]{Theorem}
\newtheorem{CorIntro}[CountAlpha]{Corollary}

\newtheorem{Thm}{Theorem}[section]
\newtheorem{Cor}[Thm]{Corollary}
\newtheorem{Lem}[Thm]{Lemma}
\newtheorem{Prop}[Thm]{Proposition}
\newtheorem{Claim}[Thm]{Claim}

\newtheorem{Setup}[Thm]{Setup}

\theoremstyle{definition}
\newtheorem{Def}[Thm]{Definition}

\newtheorem{Rk}[Thm]{Remark}

\newtheorem{Obs}[Thm]{Observation}
\newtheorem{Constr}[Thm]{Construction}

\numberwithin{equation}{section}

\newcommand{\fp}{\mathfrak{p}}

\newcommand{\fM}{\mathfrak{M}}
\newcommand{\hR}{\widehat{R}}

\newcommand{\A}{\mathbb{A}}
\newcommand{\IQ}{\mathbb{Q}}

\newcommand{\IR}{\mathbb{R}}	
\newcommand{\R}{\mathbb{R}}	
\newcommand{\IZ}{\mathbb{Z}}	
\newcommand{\cO}{\mathcal{O}}

\newcommand{\cF}{\mathcal{F}}
\newcommand{\tcF}{\widetilde{\mathcal{F}}}

\newcommand{\cJ}{\mathcal{J}}
\newcommand{\cL}{\mathcal{L}}
\newcommand{\cR}{\mathcal{R}}
\newcommand{\cS}{\mathcal{S}}
\newcommand{\cu}{\resizebox{0.16cm}{0.16cm}{$ \mathcal{U} $}}
\newcommand{\indcu}{\resizebox{0.12cm}{0.12cm}{$ \mathcal{U} $}}
\newcommand{\cv}{\resizebox{0.16cm}{0.16cm}{$ \mathcal{V} $}}
\newcommand{\indcv}{\resizebox{0.12cm}{0.12cm}{$ \mathcal{V} $}}

\newcommand{\cZ}{\mathcal{Z}}
\newcommand{\tJ}{\widetilde{J}}

\newcommand{\tM}{\widetilde{M}}
\newcommand{\tR}{\widetilde{R}}
\newcommand{\tf}{\widetilde{f}}
\newcommand{\tg}{\widetilde{g}}
\newcommand{\ts}{\widetilde{s}}
\newcommand{\tu}{\widetilde{u}}
\newcommand{\tx}{\widetilde{\mathrm{x}}}
\newcommand{\ty}{\widetilde{y}}

\newcommand{\al}{\alpha}
\newcommand{\be}{\beta}
\newcommand{\ga}{\gamma}
\newcommand{\de}{\delta}
\newcommand{\ep}{\epsilon}

\newcommand{\la}{\lambda}
\newcommand{\si}{\sigma}

\newcommand{\ini}{\mathrm{in}}
\newcommand{\Ini}{\mathrm{in}}
\newcommand{\gr}{\mathrm{gr}}

\newcommand{\ord}{\mathrm{ord}}
\newcommand{\Dir}{\mathrm{Dir}}

\newcommand{\Proj}{\mathrm{Proj}}
\newcommand{\Spec}{\mathrm{Spec}}
\newcommand{\wrt}{with respect to }

\newcommand{\poly}[3]{\Delta( \, #1  ; #2 ; #3 \, ) }
\newcommand{\cpoly}[2]{\Delta( \, #1  ; #2 \, ) }

\newcommand{\expo}{c}

\newcommand{\x}{\mathrm{x}}

\begin{document}
\title
{
	Invariance of Hironaka's characteristic polyhedron
}

\author{Vincent Cossart}
\address{Vincent Cossart\\ 
	Professeur \'em\'erite at Laboratoire de Math\'emathiques LMV UMR 8100\\ 
	Universit\'e de Versailles	\\
	45 avenue des \'Etats-Unis\\
	78035 VERSAILLES Cedex\\
	 France}
\email{cossart@math.uvsq.fr}

\author{Uwe Jannsen}
\address{Uwe Jannsen\\
	Fakult\"at f\"ur Mathematik\\
	Universit\"at Regensburg\\
	Universit\"atsstr. 31\\
	93053 Regensburg \\
	Germany
	}
\email{uwe.jannsen@mathematik.uni-regensburg.de}

\author{Bernd Schober}
\address{Bernd Schober \\
	Institut f\"ur Algebraische Geometrie\\
	Leibniz Universit\"at Hannover\\
	Welfengarten 1\\
	30167 Hannover\\
	Germany}
\email{schober@math.uni-hannover.de}

\dedicatory{To Professor Felipe Cano on the occasion of his 60th birthday}


\maketitle

\begin{abstract}
	We show that given a face of Hironaka's characteristic polyhedron, it does only depend on the singularity and a flag defined by the linear form determining the face. 
	As a consequence we get that 
	certain numerical data obtained from the characteristic polyhedron are invariants of the singularity.
	In particular, they do not depend on an embedding.  
\end{abstract}

\section*{Introduction}
\label{Intro}

Hironaka's characteristic polyhedron is an important tool for studying the local nature of a singularity defined by a non-zero ideal $ J \subset R $, where $ R $ is a regular local ring. 
For example, in \cite{CJS}, \cite{CPmixed}, \cite{CPmixed2}, \cite{InvDim2}, and \cite{HiroBowdoin}, the improvement of the singularity along a prescribed resolution process is detected using numerical data obtained from the characteristic polyhedron.
Furthermore, in \cite{SCBM}, the third author draws a connection between an invariant for resolution of singularities in characteristic zero and characteristic polyhedra of so called idealistic exponents.
(These polyhedra are closely related to Hironaka's characteristic polyhedron, see \cite{SCarthIdExp}).

Since the characteristic polyhedron is defined using an embedded situation $ J \subset R $, an essential part is to prove that the numerical data obtained from it {are} independent of the embedding and hence an invariant of the singularity $ \Spec(A) $, for $ A = R/J $.
The goal of this article is to show that, in fact, the polyhedron itself does only depend on $ A $ and the choice of elements $ (v) = (v_1, \dots, v_d ) \in A^d $ that we introduce below.

Let us be more precise.
Let $ (A, \fM , k = A/\fM ) $ be a local Noetherian ring (not necessarily regular) and put $ X := \Spec(A ) $.
Let $ ( v ) = (v_1, \ldots, v_d ) $ be a regular $ A $-sequence.
We set 
$$
	A' := A /\langle v_1, \ldots, v_d \rangle 
	\ \ \ 
	\mbox{ and }
	\ \ \ 
	X' := \Spec(A').
$$
We additionally assume that the ring of the directrix of $ X' $ at the origin coincides with the residue field $ k $.
(The directrix is a natural notion associated to the tangent cone of a scheme at a given point.
We sometimes also speak of the directrix of $ A' $, or, if we have given $ J' \subset R' $ such that $ A' = R'/J' $, we also speak of the directrix of $ J' $.
For more details and the precise definition, we refer to the appendix of this article).

Let $ (R,M, k=R/M) $ be a regular local ring and $ J \subset R $ be a non-zero ideal such that $ A \cong R/J $.
Let $ (u) = (u_1, \ldots, u_d) $ be elements in $ R $ mapping to {$ (v) =  (v_1, \ldots, v_d) $} under the canonical projection $ \pi : R  \to R/J $.
Note that the property on the directrix of $ A' $ translates to a certain condition on the directrix of $ J $ in $ R' := R / \langle u \rangle $ (see \eqref{eq:cond_Dir(J')}),
which is important to ensure that the characteristic polyhedron $\cpoly Ju $ can be explicitly computed (see section~\ref{sec:char_poly}).

We choose positive rational numbers $ \la_1, \ldots, \la_d \in \IQ_+  $, i.e., $ \la_i > 0 $ for all $ i \in \{ 1, \ldots, d \} $.
They determine a positive linear form $ \Lambda : \IR^d \to \IR $,
$$
	\Lambda(\x_1, \ldots, \x_d ) := \la_1 \x_1 + \ldots \la_d \x_d.
$$
We define 
$$
	 \delta_\Lambda := \delta_\Lambda (J; u ) := \inf \{ \Lambda(\x) \mid \x \in \cpoly Ju \}.
$$
If $ \cpoly Ju = \varnothing $, then $ \delta_\Lambda = \infty $.
	Otherwise, $ \delta_\Lambda < \infty $ and
there is a compact face $ \cF_\Lambda $ of $ \cpoly Ju $ that is defined by $ \Lambda $, 
namely,
$$
	\cF_\Lambda = \cpoly Ju \cap \{ \x \in \IR^d \mid \Lambda(\x) = \delta_\Lambda \}.
$$ 

The linear form $ \Lambda $ induces a monomial valuation on $ R $, which provides a filtration on $ R/J $ and hence a graded ring
$$
	\gr_\Lambda(R/J).
$$ 
(For more details see Definition \ref{Def:vLambda}).
The latter is an interesting object when studying singularities and their behavior along a resolution process. 

Furthermore, $ \Lambda $ (or, equivalently, the positive rational numbers $ \la_1, \ldots, \la_d $)
give rise to a flag $ F^\Lambda_{\bullet} $ in $ A $
(see section \ref{sec:flags}),
which will be crucial for our considerations.
If we have $ \la_1 \leq \la_2 \leq \ldots \leq \la_d $ and if $ \la_{\alpha(1)}, \ldots, \la_{\alpha(l)} $ are the pairwise different values (where we choose each $ \al(i) $ minimal, e.g., $ \la_{\al(1)} = \la_1 $), then
$$
\begin{array}{c} 
	F^\Lambda_\bullet = F^\Lambda_{\bullet, v} =  \{
	 F_1 \subset \ldots \subset F_l %
	 \},
	\ \  \ \	
	\mbox{for }

	F_i := V (v_{\alpha(i)}, \ldots, v_d).
\end{array}
$$

\begin{ThmIntro}
	\label{Thm:Intro}
	Let $ A $ be a local Noetherian ring.
	Let $ R $ and $ \cR $ be regular local rings and $ J \subset R $ and $ \cJ \subset \cR$ be non-zero ideals such that $ R/ J \cong \cR/\cJ \cong A $.
	Let $ \la_1, \ldots, \la_d \in \IQ_+ $ be positive rational numbers and denote by $ \Lambda : \IR^d \to \IR $ the corresponding positive linear form.
	Let $ (v) = (v_1 ,\ldots, v_d) $ and $ (\cv) = (\cv_1 ,\ldots, \cv_d) $ be elements in $ A $ whose flags in $ A $ induced by $ \Lambda $ coincide,
	$$
	F^\Lambda_{\bullet, v} = F^\Lambda_{\bullet, \indcv} 
	$$
	Let $ (u) = (u_1, \ldots, u_d ) $, resp.~$ (\cu) = (\cu_1, \ldots, \cu_d ) $, be elements in $ R $, resp.~$ \cR $, mapping to $ ( v ) $, resp.~$ (\cv) $, under the corresponding canonical projection. 
	Then 
	$$
	\delta_\Lambda (J;u) = \delta_\Lambda (\cJ;\cu)
	$$
	and, if $ \delta_\Lambda (J;u) < \infty $, then 
	$$
	\gr_\Lambda(R/J) \cong \gr_\Lambda(\cR/\cJ).
	$$
	More precisely: the isomorphism $ R/ J \cong \cR/\cJ $ respects the filtration defined by $\Lambda$. %

\end{ThmIntro}

\medskip

In other words, 
$ \delta_\Lambda (J;u) $ as well as $ \gr_\Lambda(R/J) $ do not depend on the embedding $ J \subset R $ and thus they are invariants of the singularity $ X = \Spec(A) $ and the flag $ F^\Lambda_\bullet $ in $ A $ induced by $ \Lambda $.  
Therefore, we may set 
$ \delta_\Lambda ( A , {v}) :=  \delta_\Lambda (J;u) $ and 
$ \gr_\Lambda(A) := \gr_\Lambda(R/J) $.
Furthermore, note that $ \{ \x \in \IR^d \mid \Lambda(\x) = \delta_\Lambda(A, {v}) \}$ defines a compact face of $ \cpoly \cJ\cu  $ as well as of $ \cpoly Ju $,
if $ \delta_\Lambda (A, {v}) < \infty $.

\smallskip 

The strategy for the proof is to construct a sequence of combinatorial blowing-ups that only depends on the quotient $ R/J $ and the flag $ F^\Lambda_{\bullet} $ associated to $ \Lambda $.
The important properties of the sequence will be that one can recover the number $ \delta_\Lambda(J;u) $ from its length
and that the graded ring $ \gr_\Lambda (R/J) $ naturally appears.
After defining the sequence for a given embedding $ J \subset R $,
we explain afterwards its independence of this choice.

Theorem \ref{Thm:Intro} is a generalization of the results in \cite{CosRevista}.
In contrast to the latter we are not restricted to work over the complex numbers.
(Note that $ A $ might even be of mixed characteristics). 
Further, in \cite{CosRevista} appears a condition $ (*) $ 
(in our notations: $ 0 < \la_i \leq \delta_\Lambda $, for all $ i \in \{ 1, \ldots, d \} $),
which we can overcome using \'etale coverings.

\smallskip 

Given $ \Lambda $ and $ ( u )  $ in $ R $, as before, we can define a flag $ F^\Lambda_{\bullet, u } $ in $ R $ analogously to $ F^\Lambda_{\bullet, v } $ in $ A $.
The following are immediate consequences of the theorem.

\begin{CorIntro}
	\label{Cor:Intro}
	Let $ (R,M,k) $ be a regular local ring and $ J \subset R $ be a non-zero ideal.
	Let $ ( u ) = (u_1, \ldots, u_d ) $ be a system of elements in $ R $ that can be extended to a regular system of parameters for $ R $ and such that 
	the ring of the directrix of $ J \cdot R/\langle u \rangle  $ at the origin coincides with the residue field $ k $.

	\begin{enumerate}
		\item 
		If we fix the residues $ ( v ) = (v_1, \ldots, v_d ) $ of $ ( u ) $ in $ R/J $, then the compact faces of $ \cpoly Ju $ are invariants of the singularity $ \Spec(R/J) $ and $ (v) $.
		
		\smallskip 
		
		\item 
		Let $ ( \cu ) = (\cu_1, \ldots, \cu_d) $ be another system of elements in $ R $ fulfilling the analogous conditions as $ (u ) $.
		For every linear form $ \Lambda : \IR^d \to \IR $ such that the corresponding flags 
		in $ R $ coincide, 
		$ 
			F^\Lambda_{\bullet, u } = F^\Lambda_{\bullet, \indcu } ,
		$
		we have that 
		$$ 
			\delta_\Lambda (J;u) = \delta_\Lambda (J; \cu). 
		$$
		
		\smallskip 
		
		\item 
		Let $ \Lambda_0 : \IR^d \to \IR $ be the linear form given by $ \Lambda_0(\x) = |\x| = \x_1 + \ldots + \x_d $
		(i.e., which defines the so called first face of $ \cpoly Ju $).
		Set $ \mathcal{V} := \Spec(R/M) $.
		The associated flag is trivial, 
		$ \cF^{\Lambda_0}_{\bullet} = \{ {\mathcal V} \} $
		and
		$$
			\delta(R/J) := \delta_{\Lambda_0} (R/J)
		$$ 
		is  an invariant of the singularity $ \Spec (R/J) $ and $\mathcal V$, so is $ \gr_{\Lambda_0}(R/J) $.
	\end{enumerate}		
\end{CorIntro}

Part (2) states that the faces of $ \cpoly Ju $, for which the flags associated to the defining linear forms do not change when passing from $ (u ) $ to $ (\cu) $, are stable.

The graded ring $ \gr_\delta (R/J) := \gr_{\Lambda_0} (R/J) $ associated to $ \Lambda_0 $ (or equivalently to $ \delta(R/J) $)
is a refinement of the tangent cone of $ R/J $ at the origin.
By (3), it is an invariant of the singularity $ \Spec(R/J) $ which appears in several contexts,
\cite{CPmixed2} (definitions of the invariants $\omega$ and $\epsilon$ in section~2.7)
or
\cite{InvDim2} Theorem 3.15.

\medskip

After recalling the definition of Hironaka's characteristic polyhedron in section 1, we provide the construction of a combinatorial sequence of blowing-ups arising from a linear form following \cite{CosRevista}.
In section 3, we discuss the behavior of the characteristic polyhedron under these blowing-ups.
After that we prove Theorem \ref{Thm:Intro} by connecting $ \delta_\Lambda(J;u) $ with the length of 
{a sequence of blowing-ups that is closely related to the one constructed before.}
Finally, we deduce some consequences of Theorem \ref{Thm:Intro}
	in section 5.
In the appendix we provide an algorithm to compute the directrix of a cone from its ridge.

%
%
%
%
%
%
%
%
%
%
%
%
%
%
%
%
%
%
%
%
%
%
%


\bigskip

\section{Hironaka's Characteristic Polyhedron}
\label{sec:char_poly}

First, we briefly recall the definition of the characteristic polyhedron.
In particular, we fix the setting.

\smallskip 

Let $ (R,M, k = R/M) $ be a regular local ring, 
$ \langle 0 \rangle \neq J \subset M \subset R $ a non-zero ideal,
$ ( u_1, \ldots, u_d ) $ a system of regular elements that can be extended to a regular system of parameter for $ R $.
Set $ R' = R/ \langle u \rangle $, $ M' := M/\langle u \rangle $, and $ J' = JR' $.
The {\em associated graded ring of $ R' $ at $ M' $} is defined as
$$ 
	\gr_{M'} (R')  := \bigoplus_{j \geq 0} (M')^j/ (M')^{j+1} = k[M' / (M')^2 ]. 
$$ 
and its degree one part is $ \gr_{M'}(R')_1 = M' / (M')^2 $.

The {\em initial ideal of $ J' $ \wrt $ M' $} is defined as the homogeneous ideal $ \Ini_{M'} (J') $ in $ \gr_{M'} (R') $ generated by the initial forms 
$$ 
	\ini_{M'} ( f' ) := f' \mod (M')^{n(f') + 1}, \ {\ini_{M'}(0):=0},
$$
where  {$ n(f') := \ord_{M'} ( f' ) = \sup\{ a \in \IZ_+ \mid f' \in (M')^a \} $ when $ 0 \neq f' \in J' $}, 
$$
	\Ini_{M'} (J') := \langle \ini_{M'} ( f' ) \mid f' \in J'  \rangle. 
$$

The following is an essential assumption when it comes to the characteristic polyhedron and to explicit computations of the latter:
\begin{equation}
\label{eq:cond_Dir(J')}
	\left\{ 
	\ \ \
	\begin{array}{c}
	\mbox{There is no proper $ k $-subspace $ T \subsetneq \gr_{M'}(R')_1 $ such that}
	\\[5pt]
		$$
			(\Ini_{M'} (J') \cap k[T] ) \gr_{M'} (R') = \Ini_{M'} (J').
		$$
	\end{array} \right.  
\end{equation}

\begin{Rk}
	\label{Rk:(cond_Dir(J')_equiv_V(Y') directrix}
	Let $ ( y ) = (y_1, \ldots, y_r ) $ be elements in $ R $ 
	extending $ ( u ) $ to a regular system of parameters for $ R $.
	Let $ ( y')  = (y'_1, \ldots, y'_r ) $ be their images in $ R' $.
	Condition \eqref{eq:cond_Dir(J')} states that 
	\begin{center} 
		\parbox{300pt}{ 
			$ ( Y' ) = (Y'_1, \ldots, Y'_r) $ is a minimal set of variables such that there is a system of generators for $ \Ini_{M'} (J') $ contained in $ k [Y'] $,
		}
	\end{center} 
	i.e., the directrix of $ J ' $ at $ M' $ is
	$ 
	\Dir_{M'} (J') = V( Y'_1, \ldots, Y'_r ),
	$ 
	where $ Y'_i $ denotes the image of $ y'_i $ in $ \gr_{M'} (R') $.
	(For the precise definition of the directrix, we refer to the appendix).
\end{Rk}

\begin{Def}
	Let $ (f) = (f_1, \ldots, f_m) $ be a system of elements in $ R $
	and let $ (u,y) $ be a regular system of parameters for $ R $ such that the previous condition holds.
	
	We define the {\em projected polyhedron} of $ ( f ) $ with respect to $ ( u, y ) $ as
	$$
		\poly fuy =  \poly{f_1, \ldots, f_m}uy := \mathrm{conv} \big( \bigcup\limits_{i=1}^m \poly{f_i}uy \big)
		\subset \IR^d_{\geq 0 },
	$$
	$$
	\mbox{where } \ \ \ 
	\poly{f_i}uy := 
	\mathrm{conv} 
	\left\{
	\frac{A}{n_i-|B|} + \IR^d_{\geq 0} \mid C_{A,B,i} \neq 0  
	\wedge |B| < n_i
	\right\} \subset \IR^d_{\geq 0 },
	$$
	for $ f_i = \sum_{A,B} C_{A,B,i}\,u^A\,y^B $ (finite expansion)
	with $ C_{A,B,i} \in R^\times \cup \{ 0 \} $ 
	and $ n_i := \ord_{M'} (f_i') $.
\end{Def}

In \cite{HiroCharPoly} Theorem (4.8), p.~291, Hironaka shows that
if \eqref{eq:cond_Dir(J')} holds, there exists, 
at least in $ \hR $,
 a suitable set of generators $ ( f ) = (f_1, \ldots, f_m ) $ of $ J $ 
(a so called well-prepared $( u ) $-standard basis) 
and elements $ (y) = (y_1, \ldots, y_r) $ extending $ ( u ) $ to a regular system of parameters such that
\begin{equation}
\label{eq:poly_computed} 
	\poly fuy = \cpoly Ju.
\end{equation}
For the equality, one has to use $ \cpoly{J\widehat{R}}{u} =\cpoly Ju $ (\cite{HiroCharPoly}, Lemma (4.5), p.~290). 
In \cite{CPcompl} and \cite{CSCcompl}, it is discussed under which conditions one can avoid passing to the completion to obtain the previous equality.

\smallskip 

Since \eqref{eq:poly_computed} is important later,
{we recall what it means that $ (f;u;y) $ is well-prepared.} 
For this,
{\em {let $ ( R, M, k ) $ be a regular local ring,
	let $ J \subset R $ be a non-zero ideal, and} $ ( u,y) = (u_1, \ldots, u_d; y_1, \ldots, y_r ) $ be a regular system of parameters for $ R $ such that \eqref{eq:cond_Dir(J')} holds.}
As we already remarked before, any element $ g \in R $ has a finite expansion $ g = \sum C_{A,B} \, u^A \, y^B $ in $ R $, for $ C_{A,B} \in R^\times \cup \{ 0 \} $.

We start with the definition of a $ ( u ) $-standard basis.

\begin{Def}
	\label{Def:in_0_and_u_std}
	Let $ J \subset R $ and $ ( u,y ) $ be as above.
	Let $ ( f ) = (f_1, \ldots, f_m ) $ be a system of non-zero elements in $ R $ and 
	let $ L: \IR^d \to \IR  $ be a positive linear form, $ L (\x_1, \ldots, \x_d) = a_1 \x_1 + \ldots, a_d \x_d $, for $ a_i \in \IQ_+ $.
	\begin{enumerate}
		\item 
		Let $ g = \sum C_{A,B} \, u^A \, y^B \in R \setminus \{ 0 \} $ with $ g \notin \langle u \rangle $, i.e., 
		$
			 n_{(u)}(g) := \ord_{M'}(g') < \infty.
		$	
		The {\em $ 0 $-initial form of $ g $} (\wrt $ ( u,y ))$ is defined as
		$$
			\ini_0(g) := \ini_0(g)_{(u,y)} 
			:= \sum_{|B| = n_{(u)}(g)} \overline{C_{0,B}} \, Y^B,
			\ \ \ 
			\overline{C_{0,B}} := C_{0,B} \mod M \in k. 
		$$
		
		\smallskip 
		
		\item 
		For $ g = \sum C_{A,B} \, u^A \, y^B \in R \setminus \{ 0 \} $ the {\em $ L $-valuation of $ g $},
		{\wrt $ (u, y) $, is defined as}
		$$
			{\widetilde{\nu}}_L(g) :=  {\widetilde{\nu}}_L(g)_{(u,y)} :=
			\min \{ L(A ) + |B| \mid C_{A,B} \neq 0  \}.
		$$	
		The {\em $ L $-initial form of $ g $} is defined as
		$$
			\ini_L(g) := \ini_L(g)_{(u,y)}
			:= \ini_0(g) + \sum_{A,B} \overline{C_{A,B}} \, U^A \, Y^B
			\in \gr_{{\widetilde{\nu}}_L}(R) \cong k [U, Y]
			,
		$$
		where $ (A,B) $ ranges over those elements in $ \IZ^{d+r}_{\geq 0 } $ satisfying $ L(A) + |B| = {\widetilde{\nu}}_L(g) $,
		and $ (U_1 \ldots, U_d, Y_1, \ldots, Y_r) $ denote the images of $ (u,y) $ in $ \gr_{{\widetilde{\nu}}_L}(R) $.
		Note {that} $ \ini_L (0) := 0 $.

		\smallskip 
		
		\item
		$ ( f ) $ is called a {\em $(u) $-standard basis for $ J $} if there exists a positive linear form $ L $ on $ \IR^d $ such that, for all $ 1 \leq i \leq m $, we have
		
		$(i)$ $ {\widetilde{\nu}}_L(f_i) = n_{(u)}(f_i) < \infty $,
		
		$(ii)$ $ \ini_L(f_i) = \ini_0(f_i) $,
		and 
		
		$(iii)$ $ (\ini_0(f_1), \ldots, \ini_0(f_m) ) $ is a standard basis for $ \Ini_L(J) := \langle \ini_L(g) \mid g \in J \rangle $.
	\end{enumerate}
\end{Def} 

\noindent 
For more details on these objects, we refer to \cite{CJS} section 6.

In order to define the property of a $ ( u ) $-standard basis to be normalized at a vertex, we have to introduce the notation of leading exponents.

\begin{Def}
	\label{Def:leading_exponent}
		Let $ S = k[Y_1, \ldots, Y_r] $ be a polynomial ring.
	\begin{enumerate}
		\item 
		Let $ \varphi  = \sum \lambda_B Y^B \in S $ be a homogeneous element in $ S $.
		The {\em leading exponent of $ {\varphi} $} is defined as 
		$$
			{\exp(\varphi)} := \max_{\mathrm{lex}} \{ B \mid \lambda_B \neq 0 \}.
		$$
		For a homogeneous ideal $ I \subset S $, we set
		$$
			{\exp(I)} := \{ {\exp}(\varphi) \in \IZ^r_{\geq 0} \mid \varphi \in I \mbox{ homogeneous }\}. 
		$$
		
		\smallskip 
		
		\item 
		Consider $ G_1, \ldots, G_m \in S[U_1 \ldots, U_d] $ with
		$$
			G_i = F_i(Y) + \sum P_{i,B}(U) Y^B,
			\ \ \ \
			P_{i,B}(U) \in k[U] \setminus k^\times,
		$$
		where $ F_i(Y) = \sum C_{B,i} Y^B $ is homogeneous of degree $ n_i $, for $ C_{B,i} \in k $.
		We say $ ( F_1, \ldots, F_m ) $ is {\em normalized} 
		if $ C_{B,i } = 0 $ if $ B \in {\exp}(F_1, \ldots, F_{i-1})$, for all $ i \in \{ 1, \ldots, m \} $.
		If we have additionally that $ P_{B,i} (U) \equiv 0 $ if $ B \in {\exp}(F_1, \ldots, F_{i-1})$, for all $ i \in \{ 1, \ldots, m \} $, then $ ( G_1, \ldots, G_m ) $ is called {\em normalized}.
	\end{enumerate}
\end{Def}

\begin{Def}
	\label{Def:nlzd_solv_prep_well-prep}
	Let $ J \subset R $ and $ ( u,y ) $ be as above.
	Let $ ( f ) = (f_1, \ldots, f_m ) $ be a $ ( u ) $-standard basis for $ J $
	 and let $ \x \in \poly fuy $ be a vertex of the projected polyhedron.
	 
	 \begin{enumerate}
	 	\item 
		Let $ g = \sum C_{A,B} \, u^A \, y^B \in \{ f_1, \ldots, f_m \} $.
	The {\em $ \x $-initial form of $ g $ at the vertex $ \x $} 
	(\wrt $ (u,y) $) is defined as 
	$$
	\ini_\x(g) := \ini_\x(g)_{(u,y)} 
	:=
	\ini_0(g)  + \sum_{({\star})}
	\overline{C_{A,B}}\, U^A \, Y^B,
	$$
	where the sum ranges over those exponents contributing to the vertex $ \x $,
	i.e., for which we have 
	$$ 
	\frac{A}{n_{(u)}(g) - |B|} = \x 
	\eqno{({\star})}
	$$
	We say $ (f;u;y) $ is {\em normalized at the vertex $ \x $} 
	if $ (\ini_\x (f_1), \ldots, \ini_\x(f_m)) $ is normalized.
	
	\smallskip 
	
	\item 
	The vertex $ \x \in \poly fuy $ is {called} {\em solvable} 
	if there exist $ \ga_1, \ldots, \ga_r \in k[U] $
	such that, for all $ i \in \{ 1, \ldots, m \} $,
	$$
		\ini_\x (f_i) = F_i (Y_1 + \ga_1, \ldots, Y_r + \ga_r),
		\ \ \ \
		\mbox{ where }
		F_i (Y) := \ini_0 (f_i).
	$$
	The tuple $ \ga = (\ga_1, \ldots, \ga_r) $ is then called a {\em $ v $-solution for $ (f;u;y) $}.
	
	\smallskip 
	
	\item
	The tripe $ (f;u;y) $ is {\em prepared at the vertex $ \x $}
	if $ (f;u;y) $ is normalized at $ \x $ and
	if $ \x $ is not solvable for $ (f;u;y) $.
\end{enumerate}
	The triple $ (f;u;y) $ is {called} {\em well-prepared} 
if $ (f;u;y) $ is prepared at every vertex of $ \poly fuy $.
In this case, we also say that $ (f) $ is a {\em well-prepared $ (u) $-standard basis}.
\end{Def}

For more details on these notions, we refer to \cite{CJS} section 7, or to Hironaka's original article \cite{HiroCharPoly}, and also to \cite{CSCcompl}, where several examples are discussed.

\begin{Rk}
	In the original paper \cite{CosRevista} 
	the difficulty of considering well-prepared $ ( u ) $-standard bases
	{does not appear, since the article is restricted to characteristic zero}. 
	In this case there exists the notion of ``donn\'ee distingu\'ee"
	$(f;u;y) $ that has all required properties.  
	(loc.~cit.~B.2.1, see also \cite{GiraudCharZero} Proposition 3.7 and \cite{HiroMero} p.~121 (6.5.1)--(6.5.6))
	In general, the latter do not exist.
\end{Rk}

In Lemma \ref{Lem:Preparedness_stable} below, we discuss the transformation of the projected polyhedron $ \poly fuy $ under a sequence of local blowing-ups that {are} defined by monomial map{s},
which is an immediate consequence of \cite{CJS} Lemma 9.3.
For this, let us recall:
Let $ J \subset R $ be a non-zero ideal and let $ (u,y) $ be a regular system of parameters for $ R $ such that \eqref{eq:cond_Dir(J')} holds.
The blowing-up of $ R $ in the ideal $ \langle u_1, \ldots, u_s , y \rangle  $, $ s \geq 1 $, is covered by the $ r + s $ standard charts. 
For example, if $ R \to S_1 $ is the map to the $ u_1 $-chart then
$$
	 S_1  = R \left[  \frac{u_2}{u_1}, \ldots , \frac{u_s}{u_1}, \frac{y_1}{u_1}, \ldots, \frac{y_r}{u_1}\right]
$$ 
and the map is given by the inclusion $ R \subset S_1 $.
At the origin of the $ u_1 $-chart the elements
$$ 	
	 (u',y') := \left( u_1, \frac{u_2}{u_1}, \ldots , \frac{u_s}{u_1}, u_{s+1} , \ldots, u_d, \frac{y_1}{u_1}, \ldots, \frac{y_r}{u_1} \right) 
$$
{form} a regular system of parameters.
We denote by $ R_1 $ the localization of $ S_1 $ at the maximal ideal $ \langle u', y' \rangle $.
Hence, we obtain the following local blowing-up, which is defined by a monomial map,
$$	
	\begin{array}{ccll}
		R & \longrightarrow & R_1 
		\\[3pt] 
		y_j & \mapsto & u_1 y_j', & 1 \leq j \leq r 
		\\
		u_1 & \mapsto & u_1,
		\\ 
		u_i & \mapsto & u_1 u_i', & 2 \leq i \leq s ,
		\\ 
		u_i & \mapsto & u_i, & s+1 \leq i \leq d.
	\end{array}
$$

Let $ f \in R $ with finite expansion 
$ f = \sum C_{A,B} \, u^A \, y^B $ with $ C_{A,B} \in R^\times \cup \{ 0\} $.
Suppose $ n := \ord_{M'}(f') = \ord_M(f) < \infty $ and 
$ f \in \langle u_1, \ldots, u_s, y \rangle^n $.
Then {the monomials appearing in the expansion of the strict transform of $ f $ in $ R_1 $ are given by}
$$
	u^A y^B = u_1^{A_1} \cdots u_s^{A_s} \, u_{s+1}^{A_{s+1}} \cdots u_d^{A_d} y^B
	\mapsto 
	u_1^{A_1 + \ldots + A_s + |B| - n} \, {u_2'}^{A_2} \cdots {u_s'}^{A_s} \, u_{s+1}^{A_{s+1}} \cdots u_d^{A_d} {y'}^B
	.
$$
Note that the image of $ C_{A,B} $ in $ R_1 $ remains a unit or zero.
Since $ u^A y^B $ corresponds to the point $ \frac{A}{n-|B|} \in \IR^d $ in $ \poly fuy $, this leads to a map 
$$
	 \begin{array}{rcl}
		 \Psi : \IR^d & \longrightarrow & \IR^d
		 \\
		 \x = (\x_1, \ldots, \x_s, \x_{s+1}, \ldots, \x_d)
		 & \mapsto &
		 \psi(\x) := (\x_1 + \ldots + \x_s - 1 , \x_2, \ldots, \x_s, \x_{s+1}, \ldots, \x_d) .
	 \end{array}
$$

\begin{center}
	\begin{tikzpicture}[scale=0.7]

	
	\draw (0.75,3)  coordinate (x);
	\draw (2.75,3)  coordinate (px0);
	\draw (3.75,3)  coordinate (sx);
	
	\node[above] at (1,3) {$ (\x_1, \x_2) $};
	\node[above] at (4.4,3) {$ (\x_1 + \x_2, \x_2) $};
	\node[below] at (2.75,-0.1) {$ \x_1 + \x_2-1 $};
	
	\draw[->] (0,0) -- (4,0) node[right]{$e_1$};
	\draw[->] (0,0) -- (0,4) node[above]{$e_2$};
	
	\draw[dashed, thick] (x) -- (3.75,0) -- (sx) -- (px0);
	\draw[dotted, thick] (px0) -- (2.75, -0.15);
	
	\draw[fill = lightgray] (x) circle [radius=0.09];
	\draw[fill = black] (px0) circle [radius=0.09];
	\draw[fill = white] (sx) circle [radius=0.09];
	
	\draw[|->,very thick] (5.5,2) -- (6.5,2);
	
	%
	%
	%
	%
	%
	
	\draw (8.75,3)  coordinate (x2);
	\draw (10.75,3)  coordinate (px);
	\draw (11.75,3)  coordinate (sx2);

	\node[above] at (px) {$ \Psi(\x) $};
	
	\draw[->] (8,0) -- (12,0) node[right]{$e_1$};
	\draw[->] (8,0) -- (8,4) node[above]{$e_2$};

	\draw[dotted, thick] (x2) -- (11.75,0) -- (sx2) -- (px);
	
	\draw[fill = lightgray] (x2) circle [radius=0.09];
	\draw[fill = black] (px) circle [radius=0.09];
	\draw[fill = white] (sx2) circle [radius=0.09];
\end{tikzpicture}

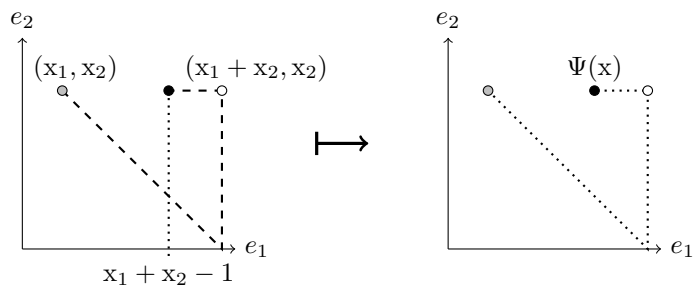
{\bf Figure 1:} Illustration of $ \Psi $ for $ d = s = 2 $ and $ (\x_1, \x_2) = (0.75,3)$.
\end{center}

\begin{Def}[\cite{CJS} Definition 2.1]
	\label{Def:permissible}
Given a locally Noetherian scheme $ X $, a reduced closed subscheme $ D \subset X $ is called {\em permissible} for $ X $ at $ x \in D $,
if $ D $ is regular at $ x $, $ D $ does not contain an irreducible component of $ X $, and if $ X $ is normally flat along $ D $ at $ x $.

We say $ D \subset X $ is permissible (or, a {\em permissible center}) for $ X $ if it is permissible at every $ x \in D $.
A sequence of blowing-ups is called permissible for $ X $ if each center is permissible for the respective strict transform of $ X $. 	
\end{Def} 

\begin{Prop}[\cite{CJS} Theorem 2.2(2)]
	\label{Prop:normal_flatness_and_generators}
If $ X = V(J) \subset \Spec(R) $ and $ D = V(\mathfrak{p}) $ for some prime ideal $ \mathfrak{p} \subset R $, 
then the normal flatness of $ X $ along $ D $ at the origin of $ \Spec(R) $ is equivalent to the existence of a standard basis $ (f) = (f_1, \ldots, f_m ) $ for $ J $ such that $ f_i \in \mathfrak{p}^{n_i} $ , for $ n_i := \ord_M(f_i) $.
\end{Prop}
	

\begin{Def}
	\label{Def:near}
		Let $ (R,M, k) $ be a regular local ring 
		and $ J \subset R $ be a non-zero ideal.
		Let $ (u) = ( u_1, \ldots, u_d ) $ be a system of elements in $ R $ such that \eqref{eq:cond_Dir(J')} holds.
		Let $ ( f ) $ be a standard basis for $ J $ and 
		let $ ( y) $ be elements extending $ (u) $ to a regular system of parameters $ (s) = (s_1, \ldots, s_q) =(u,y) $
		such that
		$
			\poly fuy = \cpoly Ju.
		$

		Set $ n_i := \ord_M(f_i)$ and $ X := \Spec (R/J) \subset Z := \Spec(R) $. 
		Let $ x \in X $ be the origin and $ \pi : Z' \to Z $ be a sequence of blowing-ups that is permissible for $ X $.
		Let $ x' \in \pi^{-1}(x) \subset Z' $ be a point lying above $ x $. 
		We denote by $ J' $ the strict transform of $ J $ in $ (R'  := \cO_{Z',x'}, M') $.
		The point $ x' $ is called {\em near to $  x $} if there exists a standard basis $ (g) = (g_1, \ldots, g_m) $ for $ J' $ such that $ \ord_{M'} (g_i) = n_i $ for all $ i \in \{ 1, \ldots, m\} $. 
	
\end{Def}

Note: This is not the original definition of a near point, as for example in \cite{CJS} Definition~2.13
	(which initially is only for a single blowing-up, but can be easily extend to a sequence).
	By \cite{CJS} Theorem 2.10 (5), (6) and (2), it follows that the above,
	which, apparently depends on choices of $(f)$ and $(s)$, is an equivalent definition. (For the reader's convenience, we point out that loc.~cit.~Definitions 1.26, 1.17, 1.1 are useful for understanding the statement of the cited theorem).

\begin{Lem}
		\label{Lem:Preparedness_stable}
		Let $ J \subset R $ and $ (u,y) $ be as before. 
		Suppose $ D = V(u_1, \ldots, u_s, y) $ is permissible for $ X = V(J) $.
		Let $ (f ) = (f_1, \ldots, f_m) $ be a standard basis for $ J $ such that $ f_i \in \langle u_1, \ldots, u_s, y\rangle^{n_i} $ for $ n_i := \ord_M (f_i) $
		and assume that $ \ord_{M\cdot R/\langle u \rangle} (f_i \cdot R/\langle u \rangle ) = \ord_{M} (f_i) $.
		
		Let $ ( f' )  = (f_1' , \ldots, f_m' ) \in R_1^m $ be the strict transforms of $ (f) $ at the origin of the $ u_1 $-chart of the blowing-up of $ R $ in $ \langle u_1, \ldots, u_s , y \rangle  $, $ s \geq 1 $.
		
		\begin{enumerate}
			\item 
			$ \poly{f'}{u'}{y'} $ is the smallest convex subset $ \Delta $ of $ \IR^d $ containing $ \Psi(\poly fuy ) $ and fulfilling $ \Delta + \IR^d_{\geq 0} = \Delta $ 
			with the convention $\Psi(\varnothing)=\varnothing$ in the case $\poly fuy=\varnothing$.
			
			\smallskip 
			
			\item
			For every vertex $ \x \in \poly{f'}{u'}{y'} $ there exists a vertex $ \x_0  \in \poly fuy $ mapping to $ \x $ under $ \Psi $.
			
			\smallskip 
			 
			\item
			Fix $ \x $ and $ \x_0 $ with $ \Psi(\x_0) = \x $.
			Then $ (f';u';y') $ is prepared at $ \x $ if and only if $(f;u;y)$ is prepared at $ \x_0 $.
		\end{enumerate}
	
		Let $ R \to R_1 \to \ldots \to R_a $ be a sequence of local blowing-ups of the above type, i.e.:
			for $ 1 \leq b \leq a - 1 $,
			if $ (u^{(b)}, y^{(b)}) $ are the parameters in $ R_b $, then
			$ R_{b+1} $ is the $ u_{i_0}^{(b)} $-chart of the blowing-up of $ R_b $ in $ \langle u_I^{(b)}, y^{(b)} \rangle $,			
			for $ i_0 \in I $ and $ (u_I^{(b)}) := (u_i^{(b)})_{i \in I } $ and some $ I \subset \{ 1, \ldots, d \} $.
			Furthermore, we require that the origins $ x_b \in \Spec(R_b) $ are near to the origin of  $ \Spec(R) $ and that each blowing-up $ R_b \to R_{b+1}$ is permissible at $ x_b $, 
			for $ 1 \leq b \leq a - 1 $.		
			(Note: the parameters in $R_{b+1}$ are chosen in the natural way described above).

			Then, taking the compositions of the functions $\Psi$, we get the same statements for the transformations of the characteristic polyhedra.
\end{Lem}
	
The result for $ R \to R_1 $ follow from the construction of the map $ \Psi $ and by \cite{CJS} Lemma 9.3 (1) and (3). 
Let us point out that the residue field does not change when passing from $ R $ to $ R_1 $. 
The extension of the result to a sequence of blowing-ups of such type is immediate.

\begin{Rk} 
	\begin{enumerate}
		\item 
		While every vertex of $ \poly{f'}{u'}{y'} $ corresponds to one of $ \poly fuy $, not every vertex of $ \poly fuy $ is mapped to a vertex of $ \poly{f'}{u'}{y'} $ under $ \Psi $.
		For example, consider $ f = y^2 + u_1^5 + u_2^7 $
		in any regular local ring with parameters $ (u_1,u_2,y) $.
		After blowing up the origin, the strict transform of $ f $ at the origin of the $ u_1 $-chart is 
		$ f' = y'^2  + u_1^3 + u_1^5 {u'_2}^7 $.
		We observe that $ \poly fuy $ has two vertex, while $ \poly{f'}{u'}{y'} $  has only one.

		\smallskip 
		
		\item 
		Note that we can also apply Lemma \ref{Lem:Preparedness_stable} for $ R_0 := R[t] $, 
		where $ t $ is an independent variable
		(e.g., for the blowing of $ R_0 $ in $ \langle t, u, y \rangle $ and the origin of the $ t $-chart).
		In Observation \ref{Obs:transform_Delta_under_S}, we use this result for a particular sequence of blowing-ups.
\end{enumerate} 
\end{Rk}

\begin{center}
	\begin{tikzpicture}[scale=0.7]
	
	\draw (0.5,4.5)  coordinate (v1);
	\draw (5,4.5)  coordinate (sv1);
	\draw (0.85,3)  coordinate (v2);
	\draw (3.85,3)  coordinate (sv2);
	\draw (2.5,1.35)  coordinate (v3);
	\draw (3.85,1.35)  coordinate (sv3); 
	\draw (5.7,0.5)  coordinate (v4); 
	\draw (6.2,0.5)  coordinate (sv4);

	\path[fill=verysoftgray] (0.5, 5) -- (v1) -- (v2) -- (v3) -- (v4) -- (7,0.5) -- (7,5) -- (0.5,5);
	\path[fill=softgray] (3.85, 5) -- (sv3) -- (sv4) -- (7,0.5) -- (7,5) -- (0.5,5);
	
	\draw[->] (0,0) -- (7,0) node[right]{$e_1$};
	\draw[->] (0,0) -- (0,5) node[above]{$e_2$};
	
	\draw[thick] (0.5, 5) -- (v1) -- (v2) -- (v3) -- (v4) -- (7,0.5);
	\draw (sv1) -- (sv2) -- (sv3) -- (sv4);
	\draw[dashed, thick] (3.85, 5) -- (sv3) -- (sv4) -- (7,0.5);

	\draw[dotted, thick] (v1) -- (5,0) -- (sv1);
	\draw[dotted, thick] (v3) -- (3.85,0) -- (sv2);
	\draw[dotted, thick] (v4) -- (6.2,0) -- (sv4);
	
	\draw[fill = white] (v1) circle [radius=0.08];
	\draw[fill = white] (sv1) circle [radius=0.08];
	\draw[fill = softgray] (v2) circle [radius=0.08];
	\draw[fill = softgray] (sv2) circle [radius=0.08];
	\draw[fill = gray] (v3) circle [radius=0.08];
	\draw[fill = gray] (sv3) circle [radius=0.08];
	\draw[fill = black] (v4) circle [radius=0.08];
	\draw[fill = black] (sv4) circle [radius=0.08];

	\end{tikzpicture}  
	
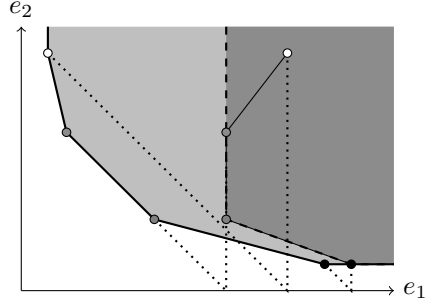
	{\bf Figure 2:} Illustration of the transform of a polyhedron $ \Delta $ (in light gray) 
	under $ \Psi $ ($ d = s = 2 $);
	the white vertex becomes a point inside of
	$ \Delta' := \Psi( \Delta) + \IR^d_{\geq 0 } $ (in gray: $ \Delta' + (1,0) $).
\end{center}

\bigskip

\section{Linear Forms and Combinatorial Sequences of Blowing Ups}

In this section we recall notations and results of \cite{CosRevista} section {\bf A.1}--{\bf A.4} that we need in the following.

\smallskip 

Let $ L : \IR^q \to \IR $ be a positive linear form on $ \IR^q $,  
$$ 
	L(\x_1, \ldots, \x_q) =a_1 \x_1 + \ldots + a_q \x_q ,
$$
for $ a_1, \ldots, a_{q} \in \IQ_+ $ positive rational numbers.
Without loss of generality, we assume that 
\begin{equation}
\label{eq:ass_L_nzld_int}
	a_1 \leq a_2 \leq \ldots \leq a_q.
\end{equation}

\begin{Obs}
	[\em Barycentric decomposition of $ L $]
	\label{Obs:bary}
	Let $  m_1 , \ldots , m_l \in \IQ_+ $, $ l = l(L) $, be the pairwise different values appearing among the coefficients $ a_1, \ldots, a_q $ of $ L $.
	We choose the indices such that they are ordered increasingly, $ m_1 <  \ldots < m_l $.
	Set $ m_0 := 0 $, and 
	$$ 
		I_k := \{ i \in \{1,\ldots, q\} \mid a_i > m_{k-1} \}, 
	$$ 
	$$
		b_k := m_k - m_{k-1},
		\ \ \ 
		\mbox{ for } 1 \leq k \leq l.
	$$
	Since the coefficients $ a_i $ are ordered increasingly, there exist indices $ \al(1), \ldots, \al(l) $ such that
	$
		I_k = \{ \al(k), \ldots, q \},
	$
	($ \al(1) = 1 $)
	and we can rewrite
	$$
	\begin{array}{rl}
		L(\x) =
		& a_1 \x_1 + \ldots a_q \x_q =
		\\[5pt]
		= 
		&
		b_1 ( \x_1 + \ldots + \x_q)
		+ b_2 ( \x_{\al(2)} + \ldots + \x_q)
		+ \ldots 
		+ b_l ( \x_{\al(l)} + \ldots + \x_q).
	\end{array}
	$$
	We have $ b_1 + \ldots + b_k = a_{\al(k)} $, for $ 1 \leq k \leq l $.
	In particular, $ b_1 = a_1 $ and $ b_1 + \ldots + b_l = a_q $.
\end{Obs}

\begin{Def}
	\label{Def:gr_L,s(R)}
	Let $ (R, M, k = R/M) $ be a regular local ring and let 
	$ (s) = (s_1, \ldots, s_q) $ be a regular system of parameters for $ R $.
	Let $ L : \IR^q \to \IR $ be a positive linear form.
	The pair $ ( L, s ) $ defines a monomial valuation  $ v_{L,s} $ on $ R $ via
	$$
		v_{L,s} ( \lambda s^\x) := L (\x),
	$$
	where $ \lambda \in R \setminus M $ is a unit and $ s^\x = s_1^{\x_1} \cdots s_q^{\x_q} $, for $ \x \in \IZ_{\geq 0}^q $.
	
	We denote by $ \gr_{L,s}(R) $
	the associated graded ring,
	$$
		\gr_{L,s}(R)  = \bigoplus_{a \in \IR} {\mathcal P}_\alpha / {\mathcal P}^+_\alpha,
	$$
	where $ {\mathcal P}_\alpha := \{ g \in R \mid v_{L,s} (g) \geq \alpha \} $ 
	and $ {\mathcal P}^+_\alpha := \{ g \in R \mid v_{L,s} (g) > \alpha \} $.
\end{Def}

Note that $ L $ takes only values in a discrete subset of $ \IR $, and hence $ \{ a \in \IR \mid {\mathcal P}_\alpha / {\mathcal P}^+_\alpha \neq 0 \} $ is a discrete subset of $ \IR $.
Furthermore, Hironaka shows in \cite{HiroCharPoly} section 1 that
$$
	\gr_{L,s} (R) \cong R/M [ \ini_{L,s} (s_1), \ldots, \ini_{L,s} (s_q)],
$$
where $ \ini_{L,s} (s_i) $ denotes  the initial form of $ s_i $ with respect to $ v_{L,s} $.

\medskip 

Given a positive linear form $ L $ with {\em integer} coefficients and fulfilling \eqref{eq:ass_L_nzld_int}, we construct a local combinatorial sequence of blowing-ups from which we later deduce the invariance of the face of $ \cpoly Ju $ corresponding to $ L $.

\begin{Def}
	\label{Def:seq_bu}
	Let $ (R,M, k = R/M) $ be a regular local ring
	and
	let $ L : \IR^q \to \IR $ be a positive linear form with coefficients $ a_1, \ldots, a_q \in \IZ_+ $ fulfilling \eqref{eq:ass_L_nzld_int}.
	We use the notations of Observation \ref{Obs:bary}.
	Let $ t $ be a new independent variable.
	We put
	$$
		R(0,0) := R[t]
	$$
	and, for $ 1 \leq i \leq l $ and $ 1 \leq j \leq b_i $,
	we define
	\begin{equation}
	\label{eq:R(i,j)}
	\left\{ \ \
	\begin{array}{ll}
	R(1,j) := R\left[
	\,
	t, 
	\, 
	\dfrac{s_{1}}{t^{j}},
	\,
	\ldots,
	\,
	\dfrac{s_{q}}{t^{j}}
	\,   
	\right],
	& i = 1,
	\\[8pt]
	R(i,j) := R\left[
	\,
	t, 
	\, 
	\dfrac{s_1}{t^{a_1}},
	\,
	\ldots,
	\,
	\dfrac{s_{\al(i)-1}}{t^{a_{\al(i)-1}}}, 
	\,
	\dfrac{s_{\al(i)}}{t^{a_{\al(i)-1}+j}},
	\,
	\ldots,
	\,
	\dfrac{s_{q}}{t^{a_{\al(i)-1}+j}}
	\,   
	\right],
	& i \geq 2.
	\end{array}
	\right. 
	\end{equation}
	Further, we let $ Z(0,0) := \Spec(R(0,0)) $ and $ Z(i,j) := \Spec(R(i,j)) $, $x(0,0)\in Z(0,0)$ the origin of $Z(0,0)$ 
	(i.e., the closed point of parameters $(t,{s})$) 
	and $x(i,j)\in Z(i,j)$  the origin of $Z(i,j)$
	(i.e., the closed point of parameters $(t, \dfrac{s_1}{t^{a_1}},\ldots,\dfrac{s_{\al(i)-1}}{t^{a_{\al(i)-1}}},\dfrac{s_{\al(i)}}{t^{a_{\al(i)-1}+j}}, \ldots,\dfrac{s_{q}}{t^{a_{\al(i)-1}+j}})$).
\end{Def}

\begin{Obs}
	\label{Obs:seq_bu}
	Let us denote by $ (i,j)_- $ the element preceding $ (i,j) $ in the set 
	$$ 
	\varepsilon := \{ (0,0), (i,j) \mid 1 \leq i \leq l, \ 1 \leq j \leq b_i \} ,
	$$
	where we order the elements of the set with respect to the lexicographical ordering.
	If we consider the blowing-up of $ Z(i,j)_- $ along 
	{the subscheme defined by}
	the ideal 
	\begin{equation}
	\label{eq:center_bu_to_(i,j)}
		 I(i,j)_- := \langle 
		 \,
		 t, 
		 \,
		 \frac{s_{\al(i)}}{t^{a_{\al(i)-1}+j-1}},
		 \,
		 \ldots,
		 \,
		 \frac{s_{q}}{t^{a_{\al(i)-1}+j-1}}
		 \,   
		 \rangle 
		 \subset R(i,j)_-,
	\end{equation}
	we observe that $ Z(i,j) $ is the affine chart that is complementary  to the divisor $ \mbox{div}(t) $.
	Therefore, the above set of rings defines a sequence of local blowing-ups,
	\begin{equation}
	\label{eq:seq_bu_for_ref}
		Z(0,0) \leftarrow 
		\cZ(1,*) \leftarrow 
		\cZ(2,*) \leftarrow 
		\ldots \leftarrow 
		\cZ(l,*),
	\end{equation} 
	$$
	\mbox{where } \   
		\cZ(i,*) \ 
		\mbox{ is an abbreviation for } \ 
		Z(i,1)
		\leftarrow \ldots 
		\leftarrow 
		Z(i,b_i),
		\ \mbox{ for }
		1 \leq i \leq l.	
	$$
	The length of \eqref{eq:seq_bu_for_ref} is $ b_1 + b_2 + \ldots + b_l = a_q$.
\end{Obs}

In order to formulate the following result in a compact way, we introduce the following notations, for given $ ( i, j ) \in \varepsilon $, $ ( i,j ) \neq (0,0 ) $,
$$
\begin{array}{rl}
	R_0 := R(i,j)_-\ ,
	& R_+ := R(i,j), \\[5pt]
	(s^-) := \left( \dfrac{s_1}{t^{a_1}},
	\dfrac{s_2}{t^{a_2}},
	\ldots,
	\dfrac{s_{\al(i)-1}}{t^{a_{\al(i)-1}}}, 
	\right),
	\\[5pt]
	(s^+) := \left( \dfrac{s_{\al(i)}}{t^{a_{\al(i)-1}+j-1}},
	\ldots,
	\dfrac{s_{q}}{t^{a_{\al(i)-1}+j-1}} \right),
	&
	(s^*) := \left( \dfrac{s^+}{t} \right) = 
	\left( \dfrac{s_{\al(i)}}{t^{a_{\al(i)-1}+j}},
	\ldots,
	\dfrac{s_{q}}{t^{a_{\al(i)-1}+j}} \right).
\end{array} 
$$
Using this, $ R_0 = R [t, s^-, s^+] $ and $ R_+ = R_0[\frac{s^+}{t}] = R [t, s^-, s^*] $.

\begin{Lem}
	\label{Lem:R(i,j)/t_polyring}
	For $ (i,j) \in \varepsilon $, $ (i,j) \neq (0,0) $, 
	the quotient $ R(i,j) / \langle t \rangle $ is a polynomial ring over the residue field $ k =R/M $.
	More precisely, the map 
	$$
	\begin{array}{rcl}
			\varphi : R(i,j) & \longrightarrow & 	k[ S^-, S^+]\\[7pt]
			a & \mapsto  & a \mod M \in k, \ \ \mbox{for } a \in R
			\\[5pt]
			s^-_\alpha & \mapsto & S^-_\alpha  := S^-_{\alpha, i, j}  := s^-_\alpha \mod t \\[5pt]
			s^*_\beta & \mapsto & S^*_\beta := S^*_{\beta, i , j} :=  s^*_\beta \mod t\\[5pt]
			t & \mapsto & 0
	\end{array}
	$$
	induces an isomorphism
	$
		R(i,j) / \langle t \rangle  
		\, \cong \,
		k[ S^-, S^*].
	$
\end{Lem}

\begin{proof}
	Since all coefficients of $ L $ are positive, the center of the first blowing-up of \eqref{eq:seq_bu_for_ref} is the origin.
	Hence, for $ (i,j) = (1,1) $, we have 
	$$
		R(1,1)  = 
		 R\left[
		 \,
		 t, 
		 \, 
		 \frac{s_{1}}{t},
		 \,
		 \ldots,
		 \,
		 \frac{s_{q}}{t}
		 \,   
		 \right].
	$$
	Using $ M = \langle s_1, \ldots, s_q \rangle $ and $ s_i = t \cdot \frac{s_i}{t} $,
	we see that $ M \cdot R_{1,1} \subset \langle t \rangle $.
	In particular, the map $ \varphi $ is well-defined and
	$$
		R(1,1) / \langle t \rangle \cong k [S_1, \ldots, S_q], 
	$$
	for $ S_i := S_{i,1,1} := \frac{s_i}{t} \mod t $, $ 1 \leq i \leq q $.
	
	Let $ (i,j) >_{\mathrm lex} (1,1) $.
	Recall the notations introduced before the lemma,
	$ R_0 = R [t, s^-, s^+] $ and $ R_+ = R_0[\frac{s^+}{t}] = R [t, s^-, s^*] $.
	Thus $ R_+ $ is the $ t $-chart of the blowing-up of $ R_0 $ along $ \langle t, s_+ \rangle $.
	By induction on $ (i,j) $, we have
	$$ 
		R_0/ \langle t \rangle 
		\cong
		k [S^-,S^+],
	$$
	where $ S^-_\alpha := s^-_\alpha \mod t $ and $ S^+_\beta := s^+_\beta \mod t $.
	In the ring $ R_+ $, we have $ s^+_\beta = t \cdot \frac{s^+_\beta}{t} $ and hence we get $ s^+_\beta \equiv 0 \mod t $.
	Therefore, if we set $ S^*_\beta := \frac{s^+_\beta}{t} \mod t $, we obtain that
	
	\smallskip 
	
	\centerline{$
	R_+/ \langle t \rangle = R_0/ \langle t \rangle[S^*] \cong k[S^-,S^*]. 	
	$}
\end{proof}

\begin{Thm}[\cite{CosRevista} {\bf A.4}, Th\'eor\`eme]
	\label{Thm:A4_CosRevista}
	Using the notations of before,
	let us denote by $ E(i,j) $ the exceptional divisor of $ Z(0,0) \leftarrow Z(i,j) $, for $ 1 \leq i \leq l $ and $ 1 \leq j \leq b_i $. 
	We have
	\begin{enumerate}
		\item 
		$ E(i,j) $ is irreducible in $ Z(i,j) $ and $ E(i,j) = V(t) $.
		
		\item
		For every $ g \in R $, we have
		$$
			\ord_{\nu(l,b_l)} (g) = v_{L,s} (g),
		$$
		where $ \nu(l,b_l) $ is the generic point of $ E(l, b_l) $.
		
		\item
		The map
		$$
			\begin{array}{rcl}
				\gr_{L,s} (R) & \to & R(l, b_{l})/ \langle t \rangle 
				\\[5pt]
				\ini_{L,s}(g) & \mapsto & t^{-\mathcal N} g \mod t,
			\end{array}
		$$
		for $ \mathcal{N} := v_{L,s} (g) $, defines an isomorphism from the associated graded ring of $ v_{L,s} $ to the ring of functions of $ E(l, b_l) $.
	\end{enumerate}
\end{Thm}


%
%
%
%
%
%
%
%
%
%
%
%
%
%

\bigskip

\section{Interpretation for the Polyhedron}

In this section we discuss the effect of the sequence of blowing-ups \eqref{eq:seq_bu_for_ref} on the characteristic polyhedron.
We fix

\begin{Setup}
\label{Setup}
	Let $ (R,M, k) $ be a regular local ring 
	and $ J \subset R $ be a non-zero ideal.
	Let $ (u) = ( u_1, \ldots, u_d ) $ be a system of elements in $ R $ such that \eqref{eq:cond_Dir(J')} holds.
	Hence the characteristic polyhedron $ \cpoly Ju \subset \IR^d $ is defined and can be computed with a suitable choice of generators $ ( f ) $ for $ J $ and elements $ ( y) $ extending $ (u) $ to a regular system of parameters $ (s) = (s_1, \ldots, s_q) =(u,y) $,
	$$
		\poly fuy = \cpoly Ju.
	$$

We fix $ \la_1, \ldots, \la_d \in \IQ_+ $ positive rational number.
Let $ \Lambda : \IR^d \to \IR $ be the positive linear form associated to these numbers,
$
	\Lambda(\x) = \Lambda(\x_1, \ldots, \x_d)  = \la_1 \x_1 + \ldots + \la_d \x_d.
$
If $ \cpoly Ju \neq \varnothing $, then $ \Lambda $ defines a compact face $ \cF_\Lambda $ of $ \cpoly Ju $,
$$
	\cF_\Lambda = \cpoly Ju \cap \{ \x \in \IR^d \mid \Lambda(\x) = \delta_\Lambda \},
$$
for $ \delta_\Lambda = \inf\{ \Lambda(\x) \mid \x \in \cpoly Ju \}  < \infty  $.
\end{Setup}

\begin{Def}
	\label{Def:vLambda}
	Let $ L : \IR^{d+r} \to \IR $ be any linear form on $ \IR^{d+r} $. 
		Then, $ L $ induces a monomial valuation on $ R $
	(in \cite{CosRevista}, the name combinatorial valuations is used),
	$$
		v_L ( \la u^A y^B ) := L(A,B) %
	\ \ \ \
	\mbox{ for } \la  \in R^\times = R \setminus M, \ A \in \IZ^d_{\geq 0}, \ B \in \IZ^r_{\geq 0}.
	$$
	This valuation induces a filtration on $  R/J \cong A $, defined by the weight 
	(it is not necessarily a valuation, for example $R/J$ may be not integral), 
	which we also denote by $ v_L $, which 
	(using the canonical projection $ \pi : R \to R/J $)
	is given by
	$$
	v_L ( \overline{g} ) := 
	v_{L,u,y} ( \overline{g} ) :=
	\sup \{
	v_L(h) \mid h \in R : \pi(h) = \overline{g},
	\}
	\ \ \ 
	\mbox{ for } \overline g \in R/J.
	$$
	We denote by $ \gr_{L}(R/J) $
	the associated graded ring,
	$$
		\gr_{L}(R/J) = \gr_{L,u,y}(R/J)  := \bigoplus_{\alpha \in \IR} {\mathcal P}_\alpha / {\mathcal P}^+_\alpha,
	$$
	where $ {\mathcal P}_\alpha := \{ \overline g \in R/J \mid v_{L} (\overline g ) \geq \alpha \} $ 
	and $ {\mathcal P}^+_\alpha := \{ \overline g \in A \mid v_{L} (\overline g )  > \alpha \} $.

	For  $ J \subset R $, $(f)$, $(u,y)$  as in Setup \ref{Setup} and a positive linear form $ \Lambda : \IR^d \to \IR $, in the case where $\cpoly Ju\not= \varnothing$, we define the linear form $ L_\Lambda : \IR^{d+r} \to \IR $ by $ L_\Lambda (A,B) := \dfrac{{\Lambda}(A)}{\delta_\Lambda(J;u)} + |B| $ and set
	$$
		\gr_\Lambda (R/J) :=  \gr_{L_\Lambda}(R/J).
	$$
	{We define $ \nu_\Lambda $ to be the valuation on $ R $ induced by $ L_\Lambda $, i.e., for $ g  = \sum C_{A,B} u^A y^B \in R \setminus \{ 0 \} $,
	\begin{equation}
		\label{eq:def_nu_Lamda}
		\nu_\Lambda (g) := \nu_\Lambda (g)_{(u,y)} := \inf \left\{ \frac{{\Lambda}(A)}{\delta_\Lambda(J;u)} + |B| 
		\mid C_{A,B} \neq 0 
		\right\}. 
	\end{equation}
	Note that this is the $ L $-valuation $ \widetilde{\nu}_L $, for $ L := \frac{1}{\delta_\Lambda} \Lambda : \IR^d \to \IR $ (Definition \ref{Def:in_0_and_u_std}(2)).} 
\end{Def}

In \cite{CosRevista}, there appears an extra assumption $(*) $
which translates in our setting to the hypothesis 
{$ 0 < \la_i \leq \delta_\Lambda $ for all $ i $.
In order to simplify proofs, we slightly sharpen this to}
$$
	0 < \la_i < \delta_\Lambda, 
	\ \ \ \mbox{ for all } i \in \{ 1, \ldots ,d \}.
	\eqno{\boldsymbol{(\ast)}}
$$ 
In  \cite{CosRevista}~\textbf{B.1.1}, there appears also the hypothesis that $k[V_1, \ldots, V_d]$ is the ring of the directrix of  $A = R/J $, which is equivalent to   $ \ord_{M'} (f') = \ord_{M} (f) $ and $ \ini_M(f)= \ini_{M'} (f') = \ini_0(f)\subset k[Y]$.
We can overcome these restrictions using the next proposition

\begin{Prop}
	\label{Prop:etale}
	Let	$ R $ and $ J $ be as before.
	Set $ A = R/J $ and let $ (v) = (v_1, \ldots, v_d) $ be a regular $ A $-sequence. 
	Let $ (u) = (u_1, \ldots, u_d ) $ be a system of elements in $ R $ that maps to $ ( v) $ under the canonical projection to $ A $.
	Consider the \'etale inclusion
	$$
	\varphi := \varphi_{R,u,c} : R \hookrightarrow S := R[w_1, \ldots, w_d]/ \langle u_1 - w_1^{\expo}, \ldots, u_d - w_d^\expo \rangle,
	$$
	where $\expo \in \IZ_+ $ is prime to $ \mbox{char}(R/M)$ when $ \mbox{char}(R/M)=p>0$.
	
	Then we have 
	
	\begin{enumerate}
		\item 
	$$
		\cpoly{J\cdot S}w= \expo \cdot \cpoly Ju.
	$$
	In particular,
	 $ \delta_\Lambda(J\cdot S) = \expo \, \delta_\Lambda(J) $, for every positive linear form defined by $ \la_1, \ldots, \la_d \in \IQ_+ $;
	 
	 \item 
	in case $\cpoly Ju\not= \varnothing$,  for every $g\in R $ and every $\Lambda$ positive linear form on $\R^d$,  $\nu_\Lambda(g)_{(u,y)} = {\nu}_\Lambda(g)_{(w,y)}$ and let $\overline{g}$, resp. $\tilde{g}$, be the residue of $g \mod  JR$, resp.$ \mod JS$, then 
	$$ \nu_\Lambda(\overline{g})=\nu_\Lambda(\tilde{g}).$$

	\end{enumerate}
\end{Prop} 

\begin{center}
	\begin{tikzpicture}[scale=0.6]
	\draw (2.5,0.5)  coordinate (u);
	\draw (1,1)  coordinate (v);
	\draw (0.5,2)  coordinate (w);
	\draw (5,1)  coordinate (2u);
	\draw (2,2)  coordinate (2v);
	\draw (1,4)  coordinate (2w);

	\draw[->] (0,0) -- (0,5) node[above]{$e_2$};
	\draw[->] (0,0) -- (7,0) node[right]{$e_1$};

	\path[fill=verysoftgray] (7,0.5) -- (u) -- (v) -- (w) -- (0.5,5) -- (7,5);
	\path[fill=softgray] (7,1) -- (2u) -- (2v) -- (2w) -- (1,5) -- (7,5);
	
	\draw[thick, dashed] (7,0.5) -- (u) -- (v) -- (w) -- (0.5,5);
	\draw[thick] (7,1) -- (2u) -- (2v) -- (2w) -- (1,5);
	
	\draw[dotted, thick] (0,0) -- (u) -- (2u);
	\draw[dotted, thick] (0,0) -- (v) -- (2v);
	\draw[dotted, thick] (0,0) -- (w) -- (2w);
	
	\draw[fill = white] (2u) circle [radius=0.09];
	\draw[fill = white] (2v) circle [radius=0.09];
	\draw[fill = white] (2w) circle [radius=0.09];
	\draw[fill = softgray] (u) circle [radius=0.09];
	\draw[fill = softgray] (v) circle [radius=0.09];
	\draw[fill = softgray] (w) circle [radius=0.09];
		
	\end{tikzpicture}  
	
	{\bf Figure 3:}
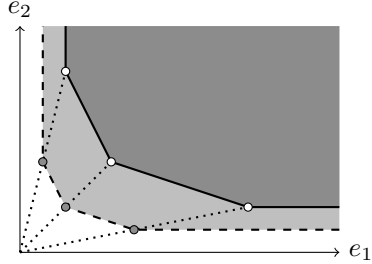
	{Example: $ d = 2 $ and $ \expo = 2 $,
		$ \cpoly Ju $ in light gray, 
		$ \cpoly{J \cdot S}w $ in gray. } 
\end{center}

\begin{proof}
	{Proof of (1)}. Let $ (f) = (f_1, \ldots, f_m ) $ and $ (y) = (y_1, \ldots, y_r) $ be elements in $ R $ such that $ (f;u;y) $ is well-prepared.
	In particular, $ \poly fuy = \cpoly Ju $.
	The equality 
	$$
		\poly{f\cdot S}wy = \expo \cdot \poly fuy
	$$	
	is obvious,
	where we abbreviate $ ( f \cdot S ) := (\varphi (f_1), \ldots, \varphi (f_m)) $.
	Furthermore, condition \eqref{eq:cond_Dir(J')} holds in $ S $ for $ ( w ) $.
	Hence, it remains to prove that $ (f \cdot S;w;y) $ remains well-prepared. 
	The $ 0 $-initial form of $ ( f \cdot S ) $ coincides with that of $ f $. 
	Since $ ( f) $ is a $ ( u ) $-standard basis for $ J $, there exists $L$  a positive linear form on $\R^d$ such that $ \ini_L(f_i)_{(w,y)} = F_i $, {for} $ 1 \leq i \leq m$. 
	Every $ g\in JS $ has an expansion $ g=\sum_{1\leq i \leq m} \phi_i f_i$, {with} $\phi_i \in S$. 
	One can see that we can choose $\phi_i$ such that in the expansion of $ \ini_L(\phi_i f_i)_{(w,y)}\in k[W,Y]$ there is no monomial
	{whose $ Y $-power is in} $ \exp(F_1, \ldots,F_{i-1})$ 
	{(Definition \ref{Def:leading_exponent}).
		Then} 
	$$ 
	\nu_L(g)_{(w,y)}= \min\{ \nu_L(\phi_i f_i)_{(w,y)}\}
	\ \ \mbox{ and } \ \  
	\ini_L(g)=\sum_{i\in \mathcal{E}} \ini_L((\phi_i f_i)_{(w,y)},
	$$ 
	where $\mathcal{E}\subset\{1,\ldots,m\}$ is the set of $ i $ such that $ \nu_L(\phi_i f_i)_{(w,y)}$ is minimal. 
	Hence $ ( f \cdot S ) $ {and} $ (w,y) $ verify the conditions of Definition~\ref{Def:in_0_and_u_std}(3): $(f)$ is a $ ( w ) $-standard basis for $ J \cdot S $.
	Further, since $ \varphi $ only changes the variables $ ( u ) $, the system $ ( f \cdot S) $ is normalized (Definition \ref{Def:nlzd_solv_prep_well-prep}(1)).

	Suppose there exists a vertex $ \x_S $ 
	in $ \poly{f\cdot S}wy $ that is solvable. 
	Then there exists a unique vertex $ \x \in \poly fuy $ that is mapped to $ \x_S $ by passing from $ R $ to $ S $, i.e., $ \x_S = \expo \cdot \x $.
	In particular, $ \ini_\x(f_i) $ is mapped to $ \ini_{\x_S} (f \cdot S) $ under the map induced by $ \varphi $ on the level of graded rings.
	But this implies that $ \x $ has to be a solvable vertex for $ \poly fuy $ since $ \x_S $ is solvable for $ \poly{f\cdot S}wy $. 
	Indeed, for every $i  \in \{ 1, \ldots, m \} $,  
	$$
		\ini_{\x_S}(f_i)=F_i (Y_1,\ldots,Y_r)+\sum_B \lambda_B Y^B W^{(n_i-\vert B\vert)\x_S}=F_i(Y_1+\ga_1 W^{\x_S},\ldots,Y_r+\ga_r W^{\x_S}),
	$$
	for $ \ga_1,\ldots,\ga_r \in {S / M_S}={R / M}$, 
	so $ \ini_{\x}(f_i)=F_i( Y+\ga U^{\x})$, $1\leq i \leq m$. 
	We claim  that $\x\in \IZ_{\geq 0 }^d$. By Giraud's construction of a space of maximal contact (\cite{GiraudMaxPos}~Definition~3.1(d)) to $ \langle \ini_\x(f_1),\cdots,\ini_\x(f_m) \rangle $, 
	there exist differential operators $P_j(D_A^{(Y)})$ which are polynomials of Hasse-Schmidt derivations in the $Y$  (Giraud calls them ``d\'eriv\'ees divis\'ees'') such that 
	$$
		s_j= P_j(D_A^{(Y)} (\ini_{\x}f_i)), 
		\ \ \ \ 1\leq j \leq e
	$$ 
	have their initial $\sigma_j=P_j(D_A^{(Y)} (F_i))$ which generate the ideal of the ridge of $(F_1,\ldots,F_m)$, and are additive homogeneous polynomials in $(Y)$, 
	\cite{GiraudMaxPos}~Lemma~1.7 and see in Appendix~\ref{app:calculdeladirectrice}. 
	So, $s_j=\sigma_j(Y+\ga U^{\x})$. The claim is then  a consequence of the hypothesis $\expo$ prime to $p=\mbox{char}(R/M)>0$  or $\mbox{char}(R/M)=0$ and the fact that $\expo \x\in \IZ_{\geq 0}^d$. 
	This contradicts the well-preparedness of $ (f; u; y) $.
	
	We conclude that $ (f\cdot S; w;y) $ is well-prepared which implies, by Hironaka's theorem (\cite{HiroCharPoly} Theorem (4.8), p.~291), that we have
	$$
	  \cpoly{J\cdot S}w = \poly{f\cdot S}wy = \expo \cdot \poly fuy = \expo \cdot  \cpoly Ju.
	$$
	The second statement is an immediate consequence.
	
	\smallskip
	Proof of (2). Take a monomial $u^Ay^B=w^{cA}y^B$,  
	 we get 
	 $$
		 \nu_\Lambda(u^Ay^B)_{(u,y)}=\dfrac{\Lambda(A)}{\delta_\Lambda(J;u)} + |B| 
		 = \dfrac{\Lambda(c A)}{c \delta_\Lambda(J;u)} + |B| =  \nu_\Lambda(w^{cA}y^B)_{(w,y)}
	 $$ 
	 This gives the first statement.
	 For any  $g\in R\setminus J$, we have a finite  expansion $g=\sum_{A,B}\lambda_{A,B}u^A y^B$ where  $\lambda_{A,B}$ invertible in $R$ and $B\not\in \exp(F_1,\ldots,F_m)$ for $\nu_\Lambda(u^Ay^B)_{(u,y)}$ minimal. 
	 It is clear that 
	 $ \nu_\Lambda(\overline{g})= \inf\{\nu_\Lambda(u^Ay^B)_{(u,y)}\mid \lambda_{A,B}\not=0\}$.
	 As  $g=\sum_{A,B}\lambda_{A,B}w^{cA} y^B$, using the first statement, we get the second one.

\end{proof}

Recall that $ A = R/J $ and that $ (v) = (v_1,\ldots, v_d) $ are the images of $ ( u ) $ under the canonical projection from $ R $ to $ R/J $,
then $ \varphi_{R,u,c} $ provides an \'etale map
$$
\varphi_{A,v,c} : A \hookrightarrow B := A[w_1, \ldots, w_d]/ \langle v_1 - w_1^{\expo}, \ldots, v_d - w_d^\expo \rangle.
$$

Proposition \ref{Prop:etale} leads to the following.

\begin{Cor}
	\label{Cor:sta_holds}
	Let the situation be as in Proposition \ref{Prop:etale}.
	After an \'etale covering, we may assume without loss of generality that, for every positive linear form $ \Lambda : \IR^d \to \IR $, we have that
	\begin{center}
		condition $(*)$ holds,
		i.e., $ 0 < \la_i < \delta_{\Lambda} $, for all $ i \in \{ 1 ,\ldots, d \} $,
		
		%
		%
	\end{center}
		and, furthermore, if $ (f) = (f_1, \ldots, f_m ) $ is a $ ( u ) $-standard basis for $ J $ and if $ ( y ) $ is a system of elements extending $ ( u ) $ to a regular system of parameters,
		then 
		\begin{center}
			$ n_{(u)} (f_i) = \ord_{M'} (f_i') = \ord_M(f_i) $ 
			and 
			
			\smallskip 
			
			$ \ini_M(f_i)= \ini_{M'} (f_i') = \ini_0(f) \in k[Y] $,
			for all $ i \in \{1, \ldots, m \} $,
			
			\smallskip 
			and $ \Dir_M(J) = V(Y_1, \ldots, Y_r) $.
		\end{center}
\end{Cor}

Let us describe the explicit algorithm that we will apply for constructing the first part of the combinatorial sequence of blowing-ups that we will use to determine $ \delta_\Lambda $.

\begin{Constr}
	We have given $ \Lambda: \IR^d \to \IR $.
	Let $ \mu  = \mu_\Lambda \in \IZ_+ $ be the lowest common multiple of the denominators of its coefficients $ \la_1, \ldots, \la_d \in \IQ_+ $.
	For any $ \rho \in \IZ_+ $, we set
	\begin{equation}
	\label{eq:def_N}
		N := N(\rho) := \rho \cdot \mu.  
	\end{equation}
	Then $ L_0 := N \cdot \Lambda $ define a positive linear form  $ L_0 : \IR^d \to \IR $ with integer coefficients,
	$$
		L_0 (\x_1, \ldots, \x_d ) =
		a_1 \x_1 + \ldots + a_d \x_d,
		\ \ \ \ \
		\mbox{ for } 
		\ 
		a_i := N \cdot \la_i \in \IZ_+,
		\ 1 \leq i \leq d.
	$$
	Without loss of generality, we may assume 
	$$ 
		a_1 \leq a_2 \leq \ldots \leq a_d 
	$$ 
	by reordering the elements $ ( u_1, \ldots, u_d ) $.
	We introduce the positive linear form $ L : \IR^q \to \IR $ (recall $ q = d + r $) by putting
	\begin{equation}
	\label{eq:def_of_L}
		L(\x_1, \ldots, \x_q) := L_0(\x_1, \ldots, \x_d) + a_d \cdot (\x_{d+1} +\ldots + \x_q).
	\end{equation}
	We define 
	$$ 
		(\cS) = (\cS_{L,\rho}) = (\cS_{\Lambda,\rho}) 
	$$ 
	to be the local sequence of blowing-ups \eqref{eq:seq_bu_for_ref} 
	corresponding to the linear form $ L $
	which we constructed in the previous section.
	Note that the dependence on $ \rho \in \IZ_+ $ comes from the definition of $ N $ \eqref{eq:def_N}.
\end{Constr}
 
\begin{Obs}
	\label{Obs:tR_and_its_coordinates}
	Since $ 0 < a_i \leq a_d $, for all $ i $, we obtain that the largest value among the coefficients of $ L $ is $ a_d $ 
	and it is achieved for $ \x_{d+1}, \ldots, \x_q $ in particular.
	The latter correspond to $ ( y_1, \ldots , y_r ) $ and hence the centers in $ (\cS) $ are always contained in the strict transform of $ V ( y ) $
	(see \eqref{eq:center_bu_to_(i,j)} and use $ (s_{r+1} , \ldots, s_q ) = (y) $).

	Recall that the largest coefficient of $ L $, here $ a_d $, corresponds to the length of the sequence of blowing-ups $ (\cS) $ (see the end of Observation \ref{Obs:seq_bu}).
	Further, using \eqref{eq:R(i,j)}, the final chart is
	$$
	\tR :=  \tR(\cS) := R(l, b(l)) = 
	R\left[
	\,
	t, 
	\, 
	\frac{u_1}{t^{a_1}},
	\, 
	\frac{u_2}{t^{a_2}},
	\,
	\ldots,
	\,
	\frac{u_{d}}{ t^{ a_d } }, 
	\,
	\frac{y_{1}}{ t^{ a_d } },
	\,
	\ldots,
	\,
	\frac{y_{r}}{ t^{ a_d } }
	\,   
	\right] = R [t, \tu, \ty ],
	$$
	$$
	\mbox{where we set }
	\ \ \  
	\tu_i := \dfrac{u_i}{t^{a_i}} 
	\ \ \mbox{ and } \ \  
	\ty_j := \dfrac{y_j}{t^{a_d}},
	\ \ \
	\mbox{ for } 1 \leq i \leq d 
	\ \mbox{ and } \
	1 \leq j \leq r.
	$$
\end{Obs}
	
	We put $ (\ts) := (\tu, \ty ) $ 
	and $ \tM := \langle t, \ts \rangle $.
	We denote by
	$ \tJ \subset \tR $ the strict transform of $ J $ in $ \tR $.
	If $ ( f) = (f_1, \ldots, f_m ) $ is a system of generators for $ J \subset R $, then we denote by $ (\tf) = (\tf_1, \ldots, \tf_m ) $ their strict transforms in $ \tR$.

\begin{Prop}
	\label{Prop:Seq_S_is_permissible}
	Let $ (R,M, k) $, $ J \subset R $, $ (s) = (u,y) $, $ \Lambda $ and $ \cF_\Lambda $ be as in Setup \ref{Setup}.
	Let $ ( f)  = (f_1, \ldots, f_m) $ be a $ ( u ) $-standard basis for $ J \subset R $ such that $ ( f;u;y) $ is well-prepared,
	i.e., $ \poly fuy = \cpoly Ju $.
	Assume that condition $(*)$ is true  and that
	\footnote{In contrast to before, we now use the index $ \ell $ instead of $ i $ in order to avoid confusion with $ ( i, j ) \in \varepsilon $.}  
	\begin{equation}
	\label{eq:assump_ord_initial}
		\ord_{M'} (f_\ell') = \ord_{M} (f_\ell) 
		\ \mbox{ and } \  
		\ini_M(f_\ell)= \ini_{M'} (f_\ell') = \ini_0(f_\ell)\in k[Y],
		\ \ \mbox{for } 1 \leq \ell \leq m.
	\end{equation} 
	Then we have
	\begin{enumerate}
		
		\item 
		The sequence of blowing-ups $ ( \cS) $ is permissible for $ X (0,0)= \Spec(R/J) $ (Definition~\ref{Def:permissible}). 
		More precisely, with the notations of Definitions~\ref{Def:in_0_and_u_std}~(1) and  \ref{Def:seq_bu}, 
		$$
			\ord_{x(0,0)}(f_{\ell})=n_{(u)}(f_\ell), 
			\ \  \ \ \ 1\leq {\ell}\leq m,
		$$ 
		and if we set $X(i,j)\subset Z(i,j)$ the strict transform of $X(0,0)$, then $x(i,j)\in X(i,j)$ for all $(i,j)$ and $x(i,j)$ are near to $x(0,0)\in X(0,0)$
		{(Definition \ref{Def:near})}.
		
		\smallskip 
		
		\item 
		The strict transforms $ (\tf ) $ of $ ( f) $ in $ \tR $
		are a $ ( \tu, t ) $-standard basis for $ \tJ $.
		
		\smallskip 
		
		\item 
		The triple $ (\tf; (\tu,t) ; \ty ) $ is well-prepared and hence 
		$ \cpoly \tJ{\tu,t} = \poly \tf{\tu,t}\ty $.
		
	\end{enumerate}	
\end{Prop}	

\begin{proof}  
	The first assertions in (1) for $x(0,0)\in X(0,0)$ are clear: the first center is $x(0,0)$ which is obviously permissible for $X(0,0)$.
	
	Note that the assumptions \eqref{eq:assump_ord_initial} imply that $ \delta > 1 $, for
	$$ 
		\delta := \delta(J;u) := \inf \{ \x_1 + \ldots + \x_d \mid (\x_1, \ldots, \x_d) \in \cpoly Ju \} .
	$$
	Results of \cite{CJS} imply that the strict transform of the $ ( u ) $-standard basis $ ( f ) $ is a $ (u',t)$-standard basis after a permissible blowing-up.
	(Here, $ (u') $ denotes the strict transform of $(u) $).
	More precisely, by \cite{CJS} Corollary 7.17, $ \delta > 1 $ implies that $ ( f )$ is a standard basis for $ J $ that is admissible for $ (u,y) $ (loc.~cit.~Definition 6.1(3)),
	and loc.~cit.~Theorem 8.1 then implies the assertion.
	
	Since we always consider the origin of the $ t $-chart along the sequence $ (\cS) $, the condition $ \delta > 1 $ remains true after each blowing-up. 
	Hence, we can apply the previous arguments after each blowing-up of $ ( \cS) $ as long as we can show that the preceding blowing-ups are permissible.  

	Since the centers appearing in $ (\cS) $ are all regular, it is sufficient to show normal flatness to get that the centers are permissible.
	Using Proposition~\ref{Prop:normal_flatness_and_generators},
	this boils down to proving that
	for the strict transforms of the given $ ( u ) $-standard basis $ ( f ) = (f_1, \ldots, f_m) $, the order of the strict transform of each $f_{\ell}$, $1\leq {\ell} \leq m$ along the ideal of the center coincides with  its order at the origin $x(i,j)$, $ ( i, j ) \in \varepsilon $ and this order is $n_{(u)}(f_{\ell})=:n_{\ell}$.

	Let $ ( i, j) \in \varepsilon $.
	Suppose that the sequence of blowing-ups until $ R(i,j) $ is permissible and that $x(i,j)\in X(i,j)$ is near to $x(0,0) $.
	If $ ( i, j ) $ is the last element in $ \varepsilon $, we are done.
	Suppose this is not the case. 
	By the above arguments the strict transforms of $ ( f ) $ in  
	$ R(i,j) $ are a standard basis for the strict transform of $ J $.
	By \eqref{eq:center_bu_to_(i,j)}, the ideal of the next center in $ (\cS) $ is
	$$
			 I(i,j) = 
			 \left\{ 
			 \begin{array}{ll}
				 \langle 
				 \,
				 t, 
				 \,
				 \dfrac{s_{\al(i)}}{t^{a_{\al(i)-1}+j}},
				 \,
				 \ldots,
				 \,
				 \dfrac{s_{q}}{t^{a_{\al(i)-1}+j}}
				 \,   
				 \rangle
				 , 
				 & 
				 \mbox{if } 
				 (i,j)_+ = (i,j+1),
				\\[10pt]
				\langle 
				\,
				t, 
				\,
				\dfrac{s_{\al(i+1)}}{t^{a_{\al(i)-1}+j}},
				\,
				\ldots,
				\,
				\dfrac{s_{q}}{t^{a_{\al(i)-1}+j}}
				\,   
				\rangle  
				,
				&
				\mbox{if } 
				(i,j)_+ = (i+1,1),
			 \end{array}
			 \right. 
	$$
	where $ ( i,j)_+ $ denotes the element in $ \varepsilon $ following $ ( i,j) $ (\wrt the lexicographical order).
	We discuss the first case and leave the second as an exercise to the reader
	which follows with the same arguments if one uses $ a_{\al(i)-1}+j = a_{\al(i)} $.
	
	Let $ u^A y^B $, $ |B| < n_\ell = \ord_M (f_\ell)$, be a monomial appearing in $ f_\ell $ with non-zero coefficient%
	, for $ \ell \in \{ 1, \ldots, m \} $.
	Recall that 
	$ R(i,j) = R\left[
	\,
	t, 
	\, 
	\dfrac{s_1}{t^{a_1}},
	\,
	\ldots,
	\,
	\dfrac{s_{\al(i)-1}}{t^{a_{\al(i)-1}}}, 
	\,
	\dfrac{s_{\al(i)}}{t^{a_{\al(i)-1}+j}},
	\,
	\ldots,
	\,
	\dfrac{s_{q}}{t^{a_{\al(i)-1}+j}}
	\,   
	\right]
	$ (see \eqref{eq:R(i,j)}).
	We observe that the number of blowing-ups until we reach $ R(i,j) $ is $ {a_{\al(i)-1}+j} $.
	Hence, in $ R(i,j) $, the monomial $ u^A y^B $ becomes 
	$  {u'}^A t^C {y'}^B $ (with the obvious notations $ ( u',y') $)
	for 
	\begin{equation}
	\label{eq:C(i,j)}
		C := C(i,j) := a_1 A_1 + \ldots + a_{\al(i)-1} A_{\al(i)-1} 
		+ (a_{\al(i)-1} + j ) [A_{\al(i)} + \ldots + A_d + |B| - n_\ell].
	\end{equation}
	
	We claim that 
	\begin{equation}
	\label{eq:C+...>n_ell}
		C(i,j) + A_{\al(i)} + \ldots + A_d + |B| > n_\ell. 
	\end{equation}
	Clearly, this would imply that the strict transform $ f_{\ell,i,j} $ of $ f_\ell $ in $ R(i,j) $ is contained in $ I(i,j)^{n_\ell} $
	and that ord$_{x(i,j)}(f_{\ell,i,j})=n_\ell$, i.e., $x(i,j)$ is near to $x(0,0)$.
	
	If we set $ a_* := a_{\al(i)-1} + j +1 $, then it is not hard to see that the claim is equivalent to 
	$$
	\frac{a_1}{a_*} A_1 + \ldots + \frac{a_{\al(i)-1}}{a_*} A_{\al(i)-1} 
	+ A_{\al(i)} + \ldots + A_d + |B|
	>
	n_\ell. 
	$$
	Using $ 0 < a_1 \leq \ldots \leq a_d $ and $ a_* \leq a_d $,
	we get that 
	$$
		\dfrac{a_1}{a_*} A_1 + \ldots + \dfrac{a_{\al(i)-1}}{a_*} A_{\al(i)-1} 
		+ A_{\al(i)} + \ldots + A_d + |B|
		\geq
			\dfrac{a_1}{a_d} A_1 + \ldots + \dfrac{a_d}{a_d} A_{d} + |B|.
	$$
	By $ (*) $, we have $ a_d = N \la_d < N \de_\Lambda $.
	Using $ N \Lambda (\x_1, \ldots, \x_d) = a_1 \x_1 + \ldots + a_d \x_d $, $ |B| < n_\ell $, and the definition of $ \de_\Lambda $, we obtain
	$$
	\frac{a_1 A_1 + \ldots + a_d A_d}{n_\ell - |B|} = 
		N\Lambda
		\left(
		\frac{A}{n_\ell - |B|}
		\right) 
		\geq N \de_\Lambda >  a_d.
	$$
	This is equivalent to 
	$$
		\dfrac{a_1}{a_d} A_1 + \ldots + \dfrac{a_d}{a_d} A_{d} + |B|
		>
		n_\ell  
	$$
	and hence shows the claim.
	
	Since the strict transforms of $ ( y ) $ are always contained in the centers of $ (\cS) $, we have $ u^A y^B \in I(i,j)^{n_\ell} $ if $ |B| \geq n_\ell $.
	Therefore, $ f_{\ell,i,j} \in  I(i,j)^{n_\ell} $ for every $ \ell \in \{ 1, \ldots, m \} $ and thus the blowing-up is permissible.
	This proves (1) and, as explained at the beginning of the proof, this also implies (2).
	
	In fact, \eqref{eq:C+...>n_ell} implies 
	\begin{equation}
	\label{eq:C>0}
		C(i,j)_+ = 
		C(i,j) + A_{\al(i)} + \ldots + A_d + |B| - n_\ell > 0 .
	\end{equation}
	So, $ C(i,j) > 0 $, for $(i,j)\not=(0,0)$ and we get
	\begin{equation}
	\label{eq:initial_mod_t}
		f_{\ell, i,j} \equiv F_\ell (Y') \mod t, \ \ \ \mbox{for all } (i,j) > (0,0),
	\end{equation} 
	where $ (y') $ are the strict transforms of $ ( y ) $ in $ R(i,j) $ and $ F_\ell(Y) = \ini_M(f_\ell) $.

	\smallskip 
	
	\noindent
	(3).
	Since the exponents in $ (y) $ do not change, $ (\tf)  $ is normalized at every vertex of $ \poly{\tf}{\tu,t}{\ty} $.
	Suppose $ (\tf; (\tu, t); \ty) $ is not well-prepared.
	Then there has to be a vertex $ \tx \in \poly{\tf}{\tu,t}{\ty} $ that is solvable.
	By Lemma \ref{Lem:R(i,j)/t_polyring}, the residue field of $ \tR $ coincides with $ k = R/M $.
	By construction of the sequence of blowing-ups,
	there exists a vertex $ \x \in \poly fuy $ mapping to $ \tx $ under $ (\cS) $ (see Observation \ref{Obs:transform_Delta_under_S} for more details)
	and by sending $ T $ to $ 1 $, we obtain the $ \x $-initial form of $ ( f ) $ from the $ \tx $-initial form of $ (\tf) $.
	Therefore, $ \tx $ being solvable implies that $ \x $ is solvable and this contradicts the well-preparedness of $ ( f;u;y) $ at $ \x $.
\end{proof}

	\begin{Obs}
		Along the sequence of blowing-ups, the power $C(i,j)$ of $ t $ starts from $ C(0,0) = 0 $, 
		is strictly increasing, then, may be stationary and may be decreasing at the end. 
		Indeed, $ C(1,1)=A_1+\ldots+A_d+\vert B\vert -n_\ell>0 $, by \eqref{eq:assump_ord_initial}.
		By \eqref{eq:C+...>n_ell}, we have, if there exists 
		$$
		i_1:= \inf \{i \in \IZ_{\geq 2} \mid  A_{\al(i)} + \ldots + A_d + |B| < n_\ell \},
		$$ 
		there is  a strict decrease  for $ (i,j)>(i_1,1) $.
		On the other hand, 
		if there exists 
		$$
		i_0:= \inf\{i  \in \IZ_{\geq 2} \mid  A_{\al(i)} + \ldots + A_d + |B| =n_\ell \},
		$$ 
		then $ C(i,j) $ is strictly increasing  for $(i,j)\leq (i_0,1)$ and is stationary between $ (i_0,1) $ and $ (i_1,1)$. 
		The last value for $C$ i.e  $C(l,b_l)$ is
		$$ 
		C(l,b_l) =
		{L(A,B)- a_d n_\ell > 0
			\ \ (\mbox{using $ L $ of } \eqref{eq:def_of_L}).},
		$$
		where the strict inequality is given by condition $(*)$. So, if there is equality in   $(*)$ as in \cite{CosRevista} at the last step of $ (\cS) $, the exponent of $t$ is $0$ and if $\delta_\Lambda<\lambda_i$, the sequence  $ (\cS) $ would stop prematurely.
	\end{Obs}

\begin{Obs}[\em Transformation of the polyhedron under $(\cS)$]
	\label{Obs:transform_Delta_under_S}
	Let $ g \in R $ be an element of a $ ( u ) $-standard basis $ (f_1, \ldots, f_m) $  computing the characteristic polyhedron, 
	$ \cpoly Ju = \poly fuy $. 
	Suppose $ n = \ord_{M'}(g) = \ord_M(g)  < \infty $. 
	Let $ u^A y^B $ be a monomial appearing in the expansion of $ g $ with $ |B| < n $.
	Hence, the corresponding point in the projected polyhedron is 
	$$	
		\x = \frac{A}{n - |B|} \in \IQ^d.
	$$
	By the previous proposition, $ (\cS) $ is permissible for $ V(g) $.
	Then the strict transform of $ u^A y^B $ in $ \tR = R(\cS) $
	(considered as a monomial in $ g $) becomes
	$$
		t^{a_1 A_1 + \ldots + a_d A_d} \, \tu^A \, 
		t^{a_d |B|} \, \ty^B \,
		t^{- a_d n}
		= \tu^A  \, t^{ C } \, 
		\ty^B 
	$$
	$$
	\mbox{for } \ \
		C := L_0(A) + a_d (|B| - n) = L(A,B)- a_d n
		\ \ (\mbox{using } \eqref{eq:def_of_L}).
	$$
	Therefore, the point $ \x \in \poly guy \subset \IQ^{d}_{\geq 0} $ maps to
	the point
	$$
		\left(\x, \frac{C}{n-|B|} \right)
		= \left(\x, \frac{L_0(A) + a_d (|B|- n)}{n-|B|} \right)
		=
		\left(\x,  L_0(\x) - a_d \right)
	$$
	in $ \poly{\tg}{\tu, t}{\ty} $, where $ \tg $ denotes the strict transform of $ g $ in $ \tR $.

	\begin{Claim}
		\label{Claim}
		We have that $ C > 0 $.
	\end{Claim}
	
	\begin{proof}
		The assertion is equivalent to the statement $ L_0(\x) - a_d > 0 $.
		Recall that $ L $ corresponds to $ \Lambda : \IR^d \to \IR $ which defines a face of $ \cpoly Ju = \poly fuy $.
		Hence we have
		$$
			\Lambda(\x) \geq \delta_\Lambda,
			\ \ \ 
			\mbox{ which is equivalent to }
			\ \ \ 
			L_0(\x) \geq N \delta_\Lambda
			\ \ \
			(\mbox{for $ N $ see \eqref{eq:def_N}}).
		$$
		By condition $(\ast)$, we have $ a_d < N \delta_\Lambda $, and therefore
		$
			L_0(\x) - a_d >  L_0(\x) -  N \delta_\Lambda \geq 0.
		$ 
	\end{proof}
\end{Obs}

Note: If $ \cpoly Ju = \varnothing $, then we also have $ \cpoly{\tJ}{\tu, t} = \varnothing $.

\begin{Def}
	If $ \cpoly Ju \neq \varnothing$, then the local sequence of blowing-ups $ (\mathcal{S}) $ provides the following map (on the level of polyhedra)
	$$ 
		\Psi : \IR^d \to \IR^{d+1}, \ \  \x \mapsto  (\x, L_0(\x)- a_d).
	$$ 
\end{Def}

As an immediate consequence of the definition of $ \Psi $ and Lemma \ref{Lem:Preparedness_stable}, we obtain

\begin{Lem}
	The restriction of $ \Psi $ to $ \cpoly Ju $ defines a well-defined map
	$$
		\psi : \cpoly Ju \to  \cpoly{\tJ}{\tu, t}.
	$$
\end{Lem}

\begin{Prop}
	\label{Prop:Poly_behavior}
	Let $ (R,M, k) $, $ J \subset R $, $ (s) = (u,y) $, $ \Lambda $ and $ \cF_\Lambda $ be as in Setup \ref{Setup}.
	Let $ ( f)  = (f_1, \ldots, f_m) $ is a set of generators for $ J \subset R $.
	\begin{enumerate}

		\item 
		If $ \x \in \cF_\Lambda $ is a vertex of the face $ \cF_\Lambda \subset \cpoly Ju $, 
		then $ \psi(\x) $ is a vertex of $ \cpoly{\tJ}{\tu, t} $.
		In fact, $ \tcF_\Lambda  := \psi(\cF_\Lambda ) \subset \cpoly{\tJ}{\tu, t} $ is the compact face of $ \cpoly{\tJ}{\tu, t} $ at which the value of $ t $ is minimal.
		More precisely, we have 
		$$
			\forall \, \x_+ = (\x_1, \ldots, \x_{d}, \x_t) \in \cpoly{\tJ}{\tu, t}  
			\ \ 
			\mbox{ we have } 
			\ \
			\x_{t} \geq N \delta_\Lambda - a_d > 0, 
		$$
		and we have equality for the points in $ \tcF_\Lambda  $.

		\smallskip 
		
		\item 
		If $ (f,u,y) $ is prepared at every vertex of the face $ \cF_\Lambda  $,
		then $ (\tf, (\tu, t), \ty) $ is prepared at every vertex of $ \tcF_\Lambda  $.

	\end{enumerate}
\end{Prop}

\begin{center}	
	\begin{tikzpicture}[scale=0.6]
	
	
	
	\draw (-3.5,0,-4.5)  coordinate (0v1);  
	\draw (-0.5,0,-2)  coordinate (0v2);    
	\draw (3,0,-1)  coordinate (0v3);  		
	\draw (5.2,0,-0.8)  coordinate (0v4);  		
	
	\draw (-3.5,3,-4.5)  coordinate (v1);  
	\draw (-0.5,2,-2)  coordinate (v2);  
	\draw (3,2,-1)  coordinate (v3);     
	\draw (5.2,4,-0.8)  coordinate (v4);   

	%
	\path[fill=softgray] (-4,0,-5) -- (0v1) -- (0v2) -- (0v3) -- (0v4) -- (6.2,0,-0.8) -- (2, 0, -5) -- (-4,0,-5);
	%
	\path[fill=lightgray] (-4,3,-5) -- (v1) -- (v2) -- (v3) -- (v4) -- (6.2,4,-0.8) -- (6.2,5,-0.8) -- (5.2,5,-0.8) -- (3,5,-1) -- (-0.5,5,-2) -- (-3.5,5,-4.5) -- (-4,5,-5) -- (-4,3,-5);
	%
	\path[fill=gray]  (-4,3,-5) -- (v1) -- (v2) -- (-3.5,2,-5) -- (-4,3,-5);
	%
	\path[fill=gray] (v3) -- (v4) -- (6.2,4,-0.8) -- (6,2,-1) -- (v3);  
	
	\draw[->] (0,0,0) -- (0,6,0) node[above]{$e_t$};
	\draw[->] (0,0,0) -- (-5,0,-5) node[left]{$e_2$};
	\draw[->] (0,0,0) -- (7,0,0) node[right]{$e_1$};
	
	\draw[thick] (-4,0,-5) -- (0v1) -- (0v2) -- (0v3) -- (0v4) -- (6.2,0,-0.8);
	
	%
	\draw[very thick] (v1) -- (v2) -- (v3) -- (v4);
	%
	\draw[very thick] (v1) -- (-3.5,5,-4.5); 
	\draw[very thick] (v2) -- (-0.5,5,-2);
	\draw[very thick] (v3) -- (3,5,-1);
	\draw[very thick] (v4) -- (5.2,5,-0.8);
	%
	\draw[very thick] (v1) -- (-4,3,-5); 
	\draw[dotted] (-4,3,-5) -- (-1,2,-2.5) ; 
	\draw[very thick] (v2) -- (-3.5,2,-5);
	\draw[very thick] (v3) -- (6,2,-1); 
	\draw[dotted] (6.2,4,-0.8) -- (4,2,-1); 
	\draw[very thick] (v4) -- (6.2,4,-0.8);
		
	\draw[thick, dotted] (0v1) -- (v1);
	\draw[thick, dotted] (0v2) -- (v2);
	\draw[thick, dotted] (0v3) -- (v3);
	\draw[thick, dotted] (0v4) -- (v4);
	
	\draw[thick, dashed] (-3.6,0,-2.8) -- (7,0,0.1);
	
	\draw[fill = white] (0v1) circle [radius=0.08];
	\draw[fill = white] (0v2) circle [radius=0.08];
	\draw[fill = white] (0v3) circle [radius=0.08];
	\draw[fill = white] (0v4) circle [radius=0.08];
	
	\draw[fill = white] (v1) circle [radius=0.08];
	\draw[fill = white] (v2) circle [radius=0.08];
	\draw[fill = white] (v3) circle [radius=0.08];
	\draw[fill = white] (v4) circle [radius=0.08];

	\node[left] at (-3.6,0,-2.8) {$ \Lambda(\x) = \delta_\Lambda $};
	\end{tikzpicture}  
	
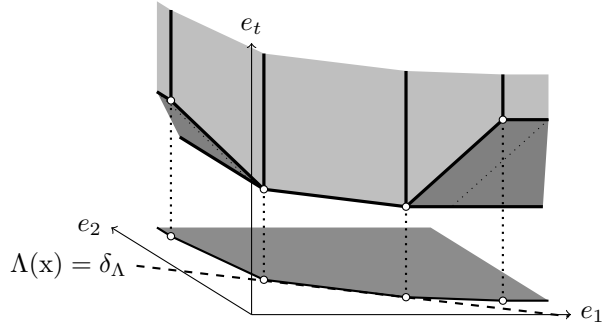
	{\bf Figure 4:} Illustration of the transformation of polyhedron under $ \psi $ for $ d = 2 $;
	
	$ \cpoly Ju $ (dim 2) in the bottom,
	$ \cpoly{\tJ}{\tu, t} $ (dim 3) above,
	the defining line of $ \cF_\Lambda $ (dashed).
\end{center} 

\begin{proof}
	The map $ \psi $ introduces a new coordinate direction (corresponding to $ t $)
	and lifts the point of $ \cpoly Ju $ along it.
	In particular,  the projection of 
	$ \cpoly{\tJ}{\tu, t} $
	to the subspace defined by the parameters $ ( u ) $
	coincides with $ \cpoly Ju $.
	This provides that $ \psi(\cF_\Lambda) $ is contained in a face of $ \cpoly{\tJ}{\tu, t} $.
	
	For $ \x \in \cF_\Lambda $, we have $ \Lambda(\x) = \delta_\Lambda $.
	For every point $ w \in \cpoly Ju \setminus \cF_\Lambda $, we have $ \Lambda(w) > \delta_\Lambda $ and hence 
	$ 
		L_0(w)-a_d > N \delta_\Lambda - a_d = L_0(\x) - a_d, 
	$
	for $ \x \in \cF_\Lambda $.
	Hence $ \psi (\cF_\Lambda) $ has to be an entire face of $ \cpoly{\tJ}{\tu, t} $.
	
	The formula in (1) follows by the proof of Claim \ref{Claim}, and
	Lemma \ref{Lem:Preparedness_stable} implies (2).
\end{proof}

\bigskip

\section{Combinatorial Sequences of Blowing ups II and Proof of Theorem A}
\label{sec:flags}

In this section we continue the construction of the combinatorial sequence and show how to obtain the number $ \delta_\Lambda $ from its length.
After drawing a connection between the linear form $ \Lambda $ and a flag on $ A = R/J $, we present the proof of Theorem \ref{Thm:Intro}.

\smallskip 

We still assume that we have the situation of Setup \ref{Setup}.
Let us recall what we did so far.
Given $ J \subset R $, $ ( u ) = (u_1, \ldots, u_d) $ such that \eqref{eq:cond_Dir(J')} holds,
and a positive linear form $ L : \IR^q  \to \IR $ (which was constructed from some positive rational numbers $ \la_1 \leq  \ldots \leq \la_d $, see \eqref{eq:def_of_L}),
we introduced a new variable $ t $ and constructed a sequence of blowing-ups $ (\cS) = (\cS_{L,\rho}) $ (see \eqref{eq:seq_bu_for_ref} and recall that $ N = \rho \cdot \mu $, for $ \rho \in \IZ_+ $, \eqref{eq:def_N}), $ R_0 = R[t] \to \ldots \to R(l, b(l))  = \widetilde{R}  $.
After a possible \'etale covering of $ R $, we achieve that $ (\mathcal{S}) $ is permissible for $ V(J) $ (Proposition~\ref{Prop:Seq_S_is_permissible}).
Finally, we discussed how the polyhedron $ \cpoly Ju $ transforms to $  \cpoly{\tJ}{\tu, t}  $ under $(\mathcal{S})$ and showed that, if $ \cpoly Ju \neq \varnothing  $,
$$
	\min\{ \tau \in \IR_{\geq 0 } \mid \exists (\x_1, \ldots, \x_d , \tau ) \in \cpoly{\tJ}{\tu, t} 
	\}
	 = N \delta_\Lambda - a_d > 0 
	 \ \ \
	 \mbox{(Proposition \ref{Prop:Poly_behavior})}.
$$
In $ \tR$, we use the regular system of parameters $ (t, \ts) = (t, \tu, \ty ) $ (see Observation \ref{Obs:tR_and_its_coordinates}).

By Proposition \ref{Prop:normal_flatness_and_generators}, we have that $ V(t, \ty) $ is a permissible center for the strict transform $ V(\tJ) $ of $ V(J) $ if and only if $  N \delta_\Lambda - a_d  \geq 1 $.

\begin{Constr}\label{Cons:S*}
	Let the situatio be as described.
	We extend the sequence of local blowing-ups $ (\cS) $ as follows:
	If $ N \delta_\Lambda - a_d < 1 $, then $ V(t, \ty ) $ is not permissible and we stop.
	
	If $ N \delta_\Lambda - a_d  \geq 1 $, then we blow up with center $ D_1 = V(t,\ty) $ and consider the origin of the $ t $-chart, i.e., we get
	$$
		R(l, b_l) \to R(l+1, 1),
	$$  
	and the map is defined by sending $ \ty_j \mapsto t \cdot \ty_j^{(1)} $, $ 1 \leq j \leq r $, and the identity otherwise.
	It is not hard to observe that we obtain 
	the characteristic polyhedron of the strict transform $ \tJ^{(1)} $ of $ \tJ $ in $ \tR_1 $ by translating $ \cpoly{\tJ}{\tu,t} $ by the vector $ (0, \ldots, 0, -1 ) \in \IR^d $.
	
	If $ V(t, \ty^{(1)}) $ is permissible for $ V(\tJ^{(1)}) $
	(which is equivalent to $  N \delta_\Lambda - a_d  \geq 2  $),
	then we blow up with center $ D_2 := V(t, \ty^{(1)}) $.
	Otherwise we stop.
	
	We continue, and, after finitely many steps, we eventually get a permissible sequence of local blowing-ups
	$$
		R[t] = R(0,0) \to R(1,1) \to \ldots \to {R(l, b_l)} 
		 \to R(l+1, 1) \to \ldots \to R(l+1, c),
		\eqno{\boldsymbol{(\mathcal{S}_*)}}
	$$
	which is of length $ a_d + c $, for 
	$$
		c := \lfloor N \delta_\Lambda - a_d \rfloor \in \IZ_{\geq 0} \cup \{ \infty \}.
	$$
	Here, $c$ may be infinite, this is the case where $ \cpoly{J}{u}=\varnothing $.
	Note: Let $ ( f ) = (f_1, \ldots, f_m ) $ be a $ ( u ) $-standard basis for $ J $ such that \eqref{eq:assump_ord_initial} holds.
	For $ i \in \{1, \ldots, c \} $, we denote by $ \tf_{\ell+1, i} $ the strict transform of $ f_\ell $ in $ R(l+1, i)$, for $ 1 \leq \ell \leq m $.
	As in \eqref{eq:initial_mod_t}, we have 
	\begin{equation}
	\label{eq:initial_mod_t_second}
	 \tf_{\ell+1, i } \equiv F_\ell (\widetilde{Y}^{(i)}) \mod t, \ \ \ \mbox{for all } i <c,
	\end{equation} 
	where $ (\ty^{(i)}) $ are the strict transforms of $ (y) $ in $ R(l+1, i) $ and $ F_\ell(Y) = \ini_M(f_\ell) $.
\end{Constr}

\begin{Rk}
		If $ \cpoly Ju = \varnothing$, the construction above provides an infinite sequence of local blowing-ups.
	In this case $ \delta_\Lambda = \infty $ and there exists a standard basis $ ( f ) = (f_1, \ldots, f_m) $ such that $ f_\ell \in \langle y \rangle^{n_\ell } $, for all $ \ell \in \{ 1, \ldots, m \} $,
	 where $ n_\ell := \ord_M(f_\ell) $.
	 Clearly, we have $ \tf_{\ell,i} \in \langle \ty^{(i)} \rangle^{n_\ell} $, for all $ \ell \in \{ 1, \ldots, m \} $ and $ i \geq 1 $.
	 Hence $ V (t, \ty^{(i)}) $ is permissible for all $ i\geq 1 $. 
\end{Rk}

	Recall that, by convention, $ a_d $ is the largest coefficient of $ L (\x) = a_1 \x_1 + \ldots + a_d \x_d $, and
	$$
		a_d =  N \lambda_d \in \IZ_+, 
	$$
	where $ \la_1, \ldots, \la_d \in \IQ_+ $ are the rational numbers determining the linear form $ \Lambda $ which defines a face $ \cF_\Lambda $ of $ \cpoly Ju $ via
	$
		\cF_\Lambda  = \cpoly Ju \cap \{ \x \in \IR^d \mid \Lambda(\x) = \delta_\Lambda \},
	$
	if $ \cpoly Ju \neq \varnothing  $.
	Furthermore, recall that by \eqref{eq:def_N},
	$$
		N = N(\rho) =  \rho \cdot \mu , 
		\ \ \ 
		\mbox{ for } \rho \in \IZ_+,
	$$
	where $ \mu $ is the lowest common multiple of the denominators of $ \la_1 , \ldots, \la_d $.

\begin{Obs}
	Putting together the previous remarks, we see that the constructed sequence of blowing-ups 
	$ (\cS_*) = (\cS_{*,L,\rho}) = (\cS_{*,\Lambda,\rho}) $ has length
	$$
		\ell en (\cS_*):= \ell en(\la_1, \ldots, \la_d, \rho ) := a_d + {c} = N \la_d + \lfloor N \delta_\Lambda - N \la_d \rfloor =  \lfloor N \delta_\Lambda  \rfloor. 
	$$
	Note that for $ \rho \in \IZ_+ $ such that $ N \delta_\Lambda \in \IZ_+ $, we have that $ \ell en(\cS_*) = N \delta_\Lambda $ and hence 
	$$ 
		\frac{\ell en(\cS_*)}{N} = \delta_\Lambda.
	$$
\end{Obs}

Since $ N = \rho \cdot \mu $, for $ \rho \in \IZ_+ $, we may send $ \rho $ to infinity and obtain as an immediate consequence of the previous observation:

\begin{Thm}
	\label{Thm:delta_from_S_*}
	Let the situation be as before.
	Then we have that
	$$
		\lim_{\rho \to \infty} \frac{\ell en (\cS_*)}{N} = \delta_\Lambda.
	$$
\end{Thm}

\noindent 
In other words, we can recover $ \delta_\Lambda $ by taking the limit over $ \rho $ of a sequence of permissible blowing-ups that only depends on the choice of $ ( \la_1, \ldots, \la_d) $ (and $ J \subset R $ as it seems).
In particular, this is also true if $ \cpoly Ju = \varnothing $.

\medskip 

We now discuss the essential step for the proof of Theorem \ref{Thm:Intro}.
More precisely, we explain how the sequence $ (\cS_*) $ can be re-constructed solely by using data in $ A = R/J $.

\begin{Def}[\cite{CosRevista} {\bf B.1.3}]
	\label{Def:Flag_for_Lambda}
	Let $ (A,\fM , k = A/\fM ) $ be a local Noetherian ring (not necessarily regular)
	and set $ X := \Spec ( A) $. 
	Let $ ( v ) = (v_1, \ldots, v_d ) $ be a regular $ A $-sequence such that 
	the ring of the directrix of $ X' = \Spec(A') $, $ A' = A / \langle v \rangle $, at $ \xi' := V (\fM \cdot A') $ is
	$ 
	k. 
	$ 

	Given a positive linear form $ \Lambda: \IR^d \to \IR $ , 
	$ \Lambda(\x)= \la_1 \x_1 + \ldots + \la_d \x_d $, $ 0 < \la_1 \leq \ldots \leq \la_d $, as before, we associate to it a flag
	$$ 
	F_\bullet^\Lambda := F_{\bullet,v}^\Lambda  := \{  %
	F_1 \subset F_2 \subset \ldots \subset F_l %
	\}
	\ \ \ \mbox{ in } X  
	$$
	as follows (using the notations of Observation \ref{Obs:bary}):
	$$
	\begin{array}{c}
	F_i := V( \, v_{\al(i)}, v_{\al(i)+1}, \ldots, v_d \, ),
	\ \ \ 
	\mbox{ for } 1 \leq i \leq l = l(\Lambda). 
	\end{array} 
	$$
	We set $ I_{F_i} := \langle v_{\al(i)}, v_{\al(i)+1}, \ldots v_d \rangle \subset A $, for $ 1 \leq i \leq l $.
\end{Def}

Clearly, the linear forms $ \Lambda $ and $ \mathcal{N} \cdot \Lambda $, for any $ \mathcal{N} \in \IZ_+ $, provide the same flag in $ A $.
\\
Furthermore, if $ t $ is an independent variable, then the flag $ F^\Lambda_\bullet $ lifts to one in $ A[t] $, 
$$ 
	F[t]^\Lambda_\bullet = \{  %
	F_1[t] \subset \ldots \subset F_l[t] 
	\}, 
\ \ \ \ \ 
\mbox{ by setting }
\ F_i [t] := V (  I_{F_i}\cdot A[t] ) .
$$

\smallskip 

\begin{Constr}[cf.~\cite{CosRevista} {\bf B.1.3}]
	\label{Constr:blow-ups_L}
	Let $ (A,\fM , k  ) $ be a Noetherian local ring (not necessarily regular)
	and let $ ( v ) = (v_1, \ldots, v_d ) $ be a regular $ A $-sequence.%

	Put $ X := \Spec ( A) $.
	Let $ \Lambda : \IR^d \to \IR $ be a positive linear form defined by $ \la_1 \leq \ldots \leq \la_d $, as before.
	We fix $ N \in \IZ_+ $ such that $ a_i := N \la_i \in \IZ_+ $.
	The barycentric decomposition of the positive linear form $ N \Lambda $ provides $ b_1, \ldots, b_{l} \in \IZ_+ $ 
(Observation \ref{Obs:bary}).
	Using the flag $ F_1 \subset F_2 \subset \ldots \subset F_l \subset F_{l + 1} := X $  
	associated to $ N \Lambda $
	(which is the same as for $ \Lambda $), we construct a sequence of blowing-ups 
	\begin{equation}
	\label{eq:first_bu_seq}
	\begin{array}{c}
	X({0,0}) \leftarrow X({1,1}) \leftarrow X({1,2}) 
	\leftarrow  \ldots \leftarrow  
	X({1,b_1})
	\leftarrow 
	X({2,1})
	\leftarrow \ldots \leftarrow 
	X({2,b_2})
	\leftarrow \ldots \\
	\ldots \leftarrow 
	X(l, 1 ) \leftarrow \ldots 
	X({l, b_{l}}) .
	\end{array}
	\end{equation}
	
	We equip the index set 
	$$ 
	\varepsilon_0(\Lambda,N) := \{ (i,j) \in \IZ_+^2 \mid 
	1 \leq i \leq l+1 \, \wedge \, j = j(i) : 1 \leq j \leq b_i \} \cup \{ (0,0) \}.
	$$
	with the lexicographical order and denote by $ (i,j)_+ $ 
	(resp.~$(i,j)_- $) 
	the element following (resp.~preceding) $ (i,j) $, as before.
	For every $ ( i,j) \in \varepsilon_0(L,N) $, we define a heptuple
	$$
	H( i,j) := \big(
	\, 
	\widetilde{X}( i,j) , \, X( i,j), \, E( i,j), \, V( i,j), \, T( i,j), \, D( i,j), \, x( i,j) 
	\big),
	$$
	\begin{tabular}{lll}
			where &
			$ \widetilde{X}( i,j) $  & is a scheme, \\
			
			&$ X( i,j) $ & is an open, affine subscheme of $ \widetilde{X}( i,j) $,\\
			
			&$ E( i,j) $ & is a divisor in $ X( i,j) $, \\
			
			&$ V( i,j) $ & is a flag $ V( i,j)_1 \subset V( i,j)_2 \subset
			\ldots \subset V( i,j)_l \subset V( i,j)_{l + 1}  = X( i,j) $,
			\\
			
			&$ T( i,j) $ & is closed subset of $ E( i,j) $ that is either empty or irreducible,
			\\
			&& if $ T( i,j) \neq \varnothing $, we denote by $ \eta( i,j) $ the generic point of $ T( i,j) $,
			\\
			
			&$ D( i,j) $ & is an irreducible curve contained in $ X( i,j) $,
			\\
			
			&$ x( i,j) $ & is a closed point contained in $ D( i,j) $.
		\end{tabular}

	\smallskip

	\noindent 
	First, we introduce a new independent variable $ t $ and set 
	$$ 
	\begin{array}{lcl}
	\widetilde{X}({0,0})  &  = & X({0,0}) := \Spec (A[t]), 
	\\
	
	E({0,0}) & :=  & V(t), 
	\\

	V({0,0})_k & := &  V ( \langle v_{\al(k)}, v_{\al(k)+1}, \ldots, v_d \rangle A[t]) = V( I_{F_k} \cdot A[t] ),\ \
	\mbox{ for } 1 \leq k \leq l,
	\\
	
	V({0,0})_{l + 1} & := &  X_{0,0},
	\\
	
	T({0,0}) & := & V ( \langle t, v_1, \ldots, v_d \rangle A[t]),
	\\
	
	D({0,0}) & := &V (\fM\cdot A[t]),
	\\
	
	x({0,0}) & := & V (\langle \fM, t \rangle A[t]) = D({0,0}) \cap E({0,0})
	\end{array}
	$$
	The first center $ Y({0,0}) $ for blowing up is the closed point $ x({0,0}) $.
	For $ (i,j) \neq (0,0) $, suppose $ H(i,j) $ is constructed, 
	then the center for the next blowing-up is defined as 
	$$
	Y(i,j) := T(i,j) \cap V(i,j)_{k(i,j)},
	$$ 
	$$
	\mbox{where } \ \ \ \ \
	k(i,j) := \left\{ 
	\begin{array}{ll}
	i, & \mbox{if } j \leq b_i - 1 ,\\
	i+1, & \mbox{if } j = b_i . \\
	\end{array}
	\right.
	$$
	
	Let us define the heptuple $ H(i,j)_+ $ for $ {(i,j)} \neq (l, b_l) $.
	Let 
	$
	\widetilde{\pi_{i,j}} : \widetilde{X}(i,j)_+ \to X(i,j) 
	$
	be the blowing-up with center $ Y(i,j) $.
	\begin{center} 
		\begin{tabular}{ll}
			$ X(i,j)_+ $ & 
			is the open affine subset of $ \widetilde{X}(i,j)_+ $ complementary to 
			\\
			
			& the strict transform of $ E(i,j) $.\\
			
			& Let $
			{\pi_{i,j}} :
			{X}(i,j)_+ \to X(i,j) 
			$
			be the restriction of $ \widetilde{\pi_{i,j}} $ to $ {X}(i,j)_+ $.
			\\
			
			$ E(i,j)_+ $ & is the exceptional divisor of $ \pi_{i,j} $, i.e., $
			E(i,j)_+ = \pi_{i,j}^{-1} (Y(i,j)),
			$ \\
			
			$ (V(i,j)_+)_k $ &
			is the strict transform of $ V(i,j)_k $, for $ 0 \leq k \leq l + 1 $. \\
			&
			In particular,
			$
			(V(i,j)_+)_{l + 1}  = X(i,j)_+
			$.
		\end{tabular}
	\end{center} 
	
	\smallskip 
	
	Let $ \ep(i,j) $ be the generic point of $ Y(i,j) $
	and $ \Gamma(i,j) $ be the directrix of $ X(i,j) $ at 
	$ \ep(i,j) $,
	$ \Gamma(i,j)  = \Dir_{\ep(i,j)} (X(i,j)) $.
	Let $ \eta(i,j)_+ $ be the generic point of the closure of  
	$ \pi_{i,j}^{-1}(\ep(i,j) ) $ in $ X(i,j)_+ $, which corresponds to the generic point of $ \Proj( \Gamma(i,j)) $ 
	(see the canonical isomorphism at the beginning of the proof for \cite{CJS} Theorem 2.14, or \cite{GiraudEtude} 1.1.2, p.II-1).
	
	\begin{center} 
		\begin{tabular}{ll}			
			$ T(i,j)_+ $ & is $ \varnothing $ if $ \eta(i,j)_+ \notin X(i,j)_+ $ and otherwise $ T(i,j)_+ = \overline{  \eta(i,j)_+ } \subset  X(i,j)_+ $.
			\\
			
			$ D(i,j)_+ $ & is the strict transform of $ D(i,j) $ in $ X(i,j)_+ $. 
			\\
			
			$ x(i,j)_+ $ & is the unique closed point contained in $ D(i,j)_+ $ and lying above $ x(i,j) $, \\
			& i.e.,
			$ x(i,j)_+ : = D(i,j)_+ \cap E(i,j)_+ $ .
		\end{tabular}
	\end{center} 
	
	\smallskip
	
	Applying this procedure, we obtain \eqref{eq:first_bu_seq}.
	We continue and extend the latter by	
	\begin{equation}
	\label{eq:second_bu_seq}
	X(l, b_{l}) \leftarrow 
	X({l + 1,1}) \leftarrow 
	\ldots 
	\leftarrow 
	X({l + 1,j}) \leftarrow 
	\ldots 
	\leftarrow 
	X({l + 1, c }).
	\end{equation}
	We set
	$
	\varepsilon(\Lambda,N) := \varepsilon_0 (\Lambda,N) \cup \{(l + 1, j ) \mid 1 \leq j \leq c \}.
	$
	With the convention $c=\infty$ when the second sequence is infinite.
	
	The candidate for the center of the first additional blowing-up is 
	$$
	Y({l ,b_l}) := T({l ,b_l}) \cap V({l ,b_l})_{l + 1 } = T({l ,b_l}).
	$$
	
	For $  (i, j ) \geq ( l ,b_l ) $, we associate to $ H({i,j}) $ a test $ OC({i,j}) $ 
	($ OC $ from French ``on continue...")
	that indicates whether $ ( i,j) $ is the last index.
	We either have $ OC({i,j}) = V $ (``vrai", i.e., $ (i,j)_+ \in \varepsilon(\Lambda,N) $, i.e., $ H({i,j})_+ $ exists)
	or $ OC({i,j}) = F $ (``faux", i.e., we stop).
	
	We define $ OC({i,j}) $ to be $ F $,
	\ \ \parbox[t]{220pt}{
		if $ x({i,j}) \notin T({i,j}) $,\\ 
		or if $ Y({i,j}) $ is not permissible for $ X({i,j}) $ at $ x({i,j}) $,\\
		or if $ Y({i,j}) $ is empty,\\
		or if $ Y({i,j}) $ is not irreducible.
	}
	
	\smallskip
	
	\noindent 
	Otherwise, we set $ OC({i,j}) = V $. 
	If $ OC({i,j}) = V $, then we construct $ H({i,j})_+ $ as above and repeat the last step 
	(i.e., we test whether $ OC(i,j)_+ = V $ and if so, we blow up).

	\rightline{\em (End of Construction \ref{Constr:blow-ups_L})}
\end{Constr}

\begin{Rk}\label{rem:directrix}
	Let $ (R,M, k) $ be a regular local ring, $ J \subset R $ be a non-zero ideal and
	$ (u,y) = ( u_1, \ldots, u_d, y_1, \ldots, y_r ) $ be a regular system of parameters for $ R $.
	As before, $ R' = R/ \langle u \rangle $, $ M' = M \cdot R $ and $ ( y') = (y_1', \ldots, y_r' )$ are the images of $ (y) $ in $ R' $.
	Further, set $ A := R/J $, $ \fM:= M \cdot A $, and let $ (v)  = (v_1, \ldots, v_d) $ be the images of $ ( u ) $ in $ A $.
	The following conditions are equivalent:
	\begin{enumerate}
		\item 
		condition \eqref{eq:cond_Dir(J')} holds for $ ( u ) $,
		i.e., 
		there is no proper $ k $-subspace $ T \subsetneq \gr_{M'}(R')_1 $ such that
		$
		(\Ini_{M'} (J') \cap k[T] ) \gr_{M'} (R') = \Ini_{M'} (J').
		$
		
		\smallskip 
		
		\item 
		the directrix of $ J' \subset R' $ at $ M' $ is $ V(Y_1', \ldots, Y_r') $.
		
		\smallskip 
		
		\item 
		the ring of the directrix of $ A' := A/ \langle v \rangle $
		at $ \fM' = \fM \cdot A' $ is $ k $.
	\end{enumerate}
	The equivalence of (1) and (2) was subject of Remark \ref{Rk:(cond_Dir(J')_equiv_V(Y') directrix}
	and  (2) $ \Leftrightarrow $ (3) is immediate.
	
	Suppose there exists a $ ( u ) $-standard basis $ ( f ) = (f_1 ,\ldots, f_m ) $ for $ J $ such that \eqref{eq:assump_ord_initial} holds true, i.e.,
	$
		\ord_{M'} (f_\ell') = \ord_{M} (f_\ell) 
	$ and $
	\ini_M(f_\ell)= \ini_{M'} (f_\ell') = \ini_0(f_\ell)\in k[Y],
	$ 
	for $ 1 \leq \ell \leq m.
	$
	Then the above is furthermore equivalent to
	\begin{enumerate}
		\item[(4)] 
		The ring of the directrix of $ A $ is $k[V_1, \ldots, V_d]$,
		where $ V_i := \ini_{\fM}(v_i)$, for $ 1 \leq i \leq d $. 
	\end{enumerate} 
\end{Rk}

\begin{Prop}\label{Prop:Same_sequences}
	Let $ (R,M, k) $ be a regular local ring, $ J \subset R $ be a non-zero ideal and
	$ (u) = ( u_1, \ldots, u_d ) $ be a system of elements in $ R $ such that \eqref{eq:cond_Dir(J')} holds.
	Let $ \Lambda : \IR^d \to \IR $ be a positive linear form defined by $ \la_1, \ldots, \la_d \in \IQ_+  $ with $ \la_1 \leq \ldots \leq \la_d $,
	and fix $ N \in \IZ_+ $ such that $ N\la_i \in \IZ_+ $.
	
	Set $ A := R/J $ and denote by $ (v) = (v_1, \ldots, v_d) $ the images of $ ( u ) $ in $ A $.
	Assume that condition $(*)$ and \eqref{eq:assump_ord_initial} hold true.
	Then the sequence of blowing-ups for $ X = \Spec(A) $ 
	obtained by \eqref{eq:first_bu_seq} and \eqref{eq:second_bu_seq}
	coincides with the restriction of $ (\cS_*) $ (in $ \Spec(R) $) to $ X \subset \Spec(R) $.
\end{Prop}

This result is the analog of \cite{CosRevista} Lemme  {\bf B.2.7.2} (there, over $ \mathbb{C} $).
	The proof is identical except at a crucial point: the argument given in \cite{CosRevista} p.~88 to compute $ \ep(i,j) $  the generic point of $ Y(i,j) $
	and  the directrix of $ X(i,j) $ at 
	$ \ep(i,j) $
	{does not work}. Indeed{,} this argument is based on the semi-continuity of the codimension of the directrix along the Samuel stratum, it is true in characteristic~0,   may be false in our general setting.  The semi-continuity  is true for the codimension of the ridge \cite{CPS}, but ridge and directrix do not coincide in general. The following lemma bridges this gap.

	\begin{Lem} 
		\label{Lem:key}
		Let $ R $ be a ring which is regular at a maximal ideal $M= \langle  u, t, y \rangle $, 
		where $ (u,t,y) = ( u_1, \ldots, u_d, t, y_1, \ldots, y_r ) $ corresponds to
	 a regular system of parameters for $ R_M $.
	 Let $ J  \subset R $ be an ideal in $ R $ and let $ (f) = (f_1, \ldots, f_m) $ be generators for $ J $ in $ R $ which form a $ ( u, t ) $-standard basis for $ J \cdot R_M $ in $ R_M $ 
	 and such that 
	(using $ R' : =  R / \langle u, t \rangle $ and $ M' = M \cdot R' $)
	\begin{equation}
	\label{eq:assump_ord_initial_modif}
		\ord_{M'} (f_\ell') = \ord_{M} (f_\ell) 
		\ \mbox{ and } \ 
		 \ini_{M'} (f_\ell') = \ini_0(f_\ell)\in k[Y],
		\ \mbox{for } 1 \leq \ell \leq m.
	\end{equation} 
	Let $Y := V (\fp ) \subset \Spec(R) $, where $ \fp = \langle u_s,\ldots,u_d, t, y_1, \ldots, y_r \rangle \subset R $, for some $ s \geq 1 $,
	and denote by  $\ep$ its generic point. 
	Assume that \eqref{eq:cond_Dir(J')} holds and 
	that the vertices $ \x = (\x_1,\ldots,\x_d, \x_t )\in \poly{f}{u,t}{y} $ with $  \x_s+\ldots+\x_d + \x_t$ minimal are prepared. 
	Let $\delta_\ep  \in \IQ_{\geq 0} $ be this minimum. 
	Assume that $ R/ \langle t \rangle \cong k[U,Y] $,
	where $ k = R/M $ and that $f_\ell \equiv F_\ell(Y) \mod t $,
	for $ F_\ell(Y) := \ini_0 (f_\ell) $ and $ 1 \leq \ell \leq m $.
	
	Then the following hold:
	
	\begin{enumerate}

		\item 
		$Y$ is permissible for $ X : = V(J) $ at $ x_0 := V(M) $ 
		if and only if $\delta_\ep \geq 1$.
		
		\item  When $\delta_\ep >1$, then
		$$
			\ini_\ep (y_1,\ldots,y_r) = I( \Dir_\ep (X))\subset \gr_\ep(R).
		$$ 
		where $ I( \Dir_\ep(X)) $ is the ideal of the directrix of $ X $ at the generic point  $\ep$ of $Y$.
		(Here, $ \Ini_\ep(.) $ denotes the initial form at the maximal ideal after localizing at $ \fp $).
		
		\item  
		When $\delta_\ep =1$, then
		\begin{equation}
		\label{eq:t_in_Idir}
			in_\ep (t, y_1,\ldots,y_r)\subset I( \Dir_\ep(X))\subset \gr_\ep(R).
		\end{equation}
		Furthermore, if $\pi: X'\to X $ is the blowing-up of $ X $ along $Y$,  let $X_+$ be the $t$-chart and $x_+$ the point lying above $ x_0 $ and on the strict transform of $ V(u,y)$.  
		Let $ \eta_+ $ be the generic point of the closure of  
	$ \pi^{-1}(\ep) $, which corresponds to the generic point of $ \Proj( \Dir_\ep(X)) $, then 
	$$
		x_+\not\in T_+
		\ \mbox{ where }
		T_+:=\overline{\eta_+ }\subset X.
	$$

	\end{enumerate}
	
	\end{Lem}
	
	\smallskip

	\begin{Rk}		
	Following \cite{CosRevista} Lemme  {\bf B.2.7.2}, the proof of Proposition~\ref{Prop:Same_sequences} is {achieved} by induction  following the sequence of blowing-ups. Thanks to  Lemma~\ref{Lem:R(i,j)/t_polyring} and equation~\eqref{eq:initial_mod_t}, the {hypothesis} of the previous lemma, {holds} at each blowing-up, except for the first one which is centered at the origin. The ``explicitation'' \cite{CosRevista}~{{\bf B.2.8}--{\bf B.2.10}} follows  straightforwardly. 
\end{Rk}

\begin{proof}[Proof of Lemma \ref{Lem:key}]
	We first prove (1) and (2). 
	By localizing $ R $ at $ \fp $, we may assume without loss of generality that $ Y = x_0 $ is the closed point and $ \fp = M $.
	Then $\delta_\ep$ is the usual $\delta$ of Corollary~\ref{Cor:Intro}(3).
	Note that then the condition on the preparedness is that all vertices on the face defined by $  \x_1+ \ldots + \x_d + \x_t $ minimal are prepared. 
	
	For $ \ell \in \{ 1, \ldots, m \} $, 
	we expand $ f_\ell  = \sum C_{A,B,d,\ell}\,  u^A\, t^d \,y^B $, $ C_{A,B,d,\ell} \in R^\times \cup \{ 0 \} $ 
	and put $ n_\ell := \ord_M (f_\ell) $. 
	We have that $ f_\ell \in M^{n_\ell} $ if and only if 
	$  |A|+d +|B| \geq n_\ell $, for all $ (A,B,d) $ with $ C_{A,B,d,\ell} \neq 0 $,
	which is further equivalent to $ \delta_\ep \geq 1 $.

	Suppose $ \delta_\ep > 1 $.
	Then \cite{CJS} Lemma 7.4(2) implies that $ \ini_M (f_\ell) = \ini_0 (f_\ell) \in k[Y] $, for all $ \ell $.
	As in the proof of Proposition \ref{Prop:Seq_S_is_permissible}, \cite{CJS} Corollary  7.17 and 
	Proposition \ref{Prop:normal_flatness_and_generators}
	provide that the blowing-up is permissible.
	
	On the other hand, assume $ \delta_\ep = 1 $. 
	By \cite{CJS} Lemma 7.4(3), we have $ \ini_M(f_\ell) = \ini_{\delta_\ep} (f_\ell) $, for all $ \ell $.
	Note that $ (f) $ being a $ (u,t) $-standard basis implies in the given situation that $ ( f) $ is also a standard  basis for the ideal.
	Hence Proposition \ref{Prop:normal_flatness_and_generators} and the above yield that the blowing-up is permissible.
	This completes the proof for (1).
	
	\smallskip 
	
	Furthermore, $ \delta_\ep > 1 $ provides that the initial at $ \epsilon $ is $ \ini_0 (f_\ell) $ for all $ \ell $.
	By \eqref{eq:cond_Dir(J')} and \cite{CJS}~Lemma~1.10(3)(i), part (2) of the lemma follows.
	
	\smallskip 
	
	It remains to prove part (3). 
	As $x_+$ is not on the strict transform of div$(t)$, \eqref{eq:t_in_Idir} implies the rest of part (3).
	By definition, $ \gr_\ep ( R) $ is the graded ring at the maximal ideal after localizing $ R $ at $ \fp $ and hence
	$$
		\gr_\ep(R) \cong k(U_1,\ldots,U_{s-1})\big[\ini_\ep (u_s),\ldots,\ini_\ep (u_d), \ini_\ep (t),\ini_\ep (y_1), \ldots, \ini_\ep (y_r) \big] ,
	$$
	where $(U_1,\ldots,U_{s-1})=(u_1,\ldots,u_{s-1}) \mod tR $ (isomorphism of the hypothesis).

	The hypothesis  $f_\ell \equiv F_\ell(Y)$~mod~$t$ provides that the initial form of $ f_\ell $ at $ \ep $ coincides with the $ 0 $-initial form of $ f_\ell $ modulo $ tR $,
	$$
		\ini_\ep (f_\ell ) \equiv F_\ell (Y) \mod t R,
	$$
	and, by  \cite{CJS}~Lemma~1.10(3)(i), we have
	$$ %
		\ini_\ep (y_1,\ldots,y_r)\subset I( \Dir_\ep(X)) \mod \langle \ini_\ep (t) \rangle 
	$$ %
	Then, by the appendix~\ref{app:calculdeladirectrice}, $I( \Dir_\ep(X))$ is computed by applying to  the initials of $f_\ell $ the following:  extraction of $p$-roots, Hasse-Schmidt derivations and derivations with respect to an absolute $p$-basis of $k(U_1,\ldots,U_{s-1})$, which may be chosen as an absolute $p$-basis of the field $k$ completed with $(U_1,\ldots,U_{s-1})$. 
	Noting that     
	$$
	\begin{array}{c} 
		\ini_\ep(f_\ell)=F_\ell(\ini_\ep(y))+\ini_\ep(t) \cdot \sum_{|B| < n_\ell} G_{\ell,B} \, \ini_\ep(y)^B, \\[5pt]
		\mbox{for some}\  G_ {\ell,B} \in  {k(U_1,\ldots,U_{s-1})}[ \ini_\ep(u_s),\ldots,\ini_\ep(u_d), \ini_\ep (t) ]	,
	\end{array}
	$$ 
	we get for $1\leq j \leq r$:
	\begin{equation}\label{eq:t_in_Idir?_2}
		\ini_\ep (y_j)+\ini_\ep(t)P_j \in I( \Dir_\ep(X)),
		\ \ 
		P_i\in {k(U_1,\ldots,U_{s-1})}[ \ini_\ep(u_s),\ldots,\ini_\ep(u_d), \ini_\ep (t)]	
	\end{equation}
	
	Suppose that  \eqref{eq:t_in_Idir} is wrong: by \eqref{eq:t_in_Idir?_2}, 
	$F_\ell(\ini_\ep (y)+\ini_\ep(t)P)\in k(U_1,\ldots,U_{s-1})[V]$,
	where $  k(U_1,\ldots,U_{s-1})[V] $ is the smallest $ k(U_1,\ldots,U_{s-1})$-algebra containing generators of the ideal $I( \Dir_\ep(X))$.
	Note that for all $ \ell$, 
	$1\leq \ell \leq m$, 
	$$
		F_\ell (\ini_\ep (y)+\ini_\ep(t)P)-\ini_\ep(f_\ell)\in \ini_\ep(t) \cdot k(U_1,\ldots,U_{s-1})[V],
	$$ 
	hence all difference must be equal to $ 0$, 
	else $\ini_\ep(t)\in k(U_1,\ldots,U_{s-1})[V]$. 
	Therefore: 
	$$ %
		\langle 
			\ini_\ep (y_1)+\ini_\ep(t)P_1,\ldots, \ini_\ep (y_r)+\ini_\ep(t)P_r
		\rangle 
		=I( \Dir_\ep(X)).
	$$ %
  We can define $ z_j\in R $, for $1\leq j \leq r$, with $\ini_\ep(z_j)= \ini_\ep (y_j)+\ini_\ep(t)P_j$. 
  Then the vertices of  $ \poly{f}{u,t}{z} $ verify $\x_s+\ldots+\x_d+\x_t>1$, 
  which contradicts the preparation of the vertices of $ \poly{f}{u,t}{y} $ with $\x_s+\ldots+\x_d + \x_t =\delta_\ep=1$.	
\end{proof}

Finally, let us discuss the connection between the ring of functions of the divisors $ \mathrm{div}(t) $ and the graded rings.

\begin{Obs}
	Let the situation be as in Proposition \ref{Prop:Same_sequences}.
	Let $ ( y ) = (y_1, \ldots, y_r) $ be elements in $ R $ such that $ ( u, y ) $ is a regular system of parameters for $ R $.
	Let $ (f) = (f_1, \ldots, f_m) $ be a $ ( u ) $-standard basis for $ J $
	such that  \eqref{eq:assump_ord_initial} holds.
	Denote by $ f_{\ell,*} $ the strict transform of $ f_\ell $ in $ \tR_* $ (i.e., at the end of the sequence $(\cS_*)$)
	and recall that $ n_\ell := \ord_M(f_\ell) $, for $ 1 \leq \ell \leq m $.
	Suppose that $ \cpoly Ju \neq \varnothing $.

	\medskip 
	 
	 \noindent 
	{\bf (1)} 
	If $ N \delta_\Lambda \notin \IZ $ is not an integer, then also $ N \delta_\Lambda - a_d \notin \IZ $ is not.
	In particular, we have $ 0 < N \delta_\Lambda - a_d - K  < 1 $.
	Recall that the latter corresponds to the smallest power of $ t $ appearing in the elements $ ( f ) $ and contributing to the polyhedron (i.e., appearing in those monomials of $ f_{\ell,*} $ whose $ y $-power is $ B $ with $ |B| < n_\ell $).
	From this, we obtain that
	$$
		f_{\ell,*} \equiv F_\ell(Y_*) \mod t, \ \ \  \ \ \mbox{for all } \ \ell \in \{ 1, \ldots, m \},
	$$
	where $ (y_*) $ are the strict transforms of $ (y) $ in $ R_* $
	and $ F_\ell(Y) = \ini_M(f_\ell) \in k[Y_1, \ldots, Y_r] $.
		
	Hence, if {we} denote by $ h_* $ the strict transform of an element $ h \in R $ under $ (\cS_* )$, then the map
	$$
	\begin{array}{rcccl}
		R & \to & R_* & \to & R_*/\langle t \rangle
		\\ 
		h & \mapsto & h_* & \mapsto & h_* \mod t,
	\end{array}
	$$
	induces an isomorphism from the graded ring of $ A $ at $ \fM $ to the ring of functions of the divisor  $ \mathrm{div}( t )$,
	$$
		\gr_\fM(A)  \longrightarrow A/ \langle t \rangle.
	$$ 
	Note that $ \gr_\fM(A) \cong \gr_M(R)/ \Ini_M(J)  = \gr_M(R)/ \langle F_1, \ldots, F_m \rangle $.
	
	\medskip 
	
	\noindent
	{\bf (2)}
	Let us consider the case $ N \delta_\Lambda \in \IZ_+ $.
	Using $ \Lambda : \IR^d \to \IR $, we define a positive linear form $ \cL : \IR^d \times \IR^r \to \IR $ by
	$$ 
	\cL (\x_1, \ldots, \x_e, \x_{e+1}, \ldots, \x_r ) := 
	\frac{\Lambda(\x_1, \ldots, \x_e)}{\delta_\Lambda} + | (\x_{e+1}, \ldots, \x_{e+r})|.
	$$
	Note that the $ \cL$-initial form of $ f_\ell $ coincides with the initial form of $ f_{\ell} $ along the face $ \cF_\Lambda $ of $ \cpoly Ju $ defined by $ \Lambda $.
	(The latter is defined as the $ 0 $-initial form plus the sum over those monomials for which the corresponding point in the projected polyhedron of $ ( f) $ \wrt $ ( u; y ) $ lie on the face $ \cF_\Lambda $).
	
	As we have explained in Definition \ref{Def:vLambda}, $ \cL $ induces a filtration $ v_{\cL,u,y} $ on $ A = R/J $ and the corresponding graded ring is denoted by $ \gr_\cL(A) = \gr_{\cL,u,y} (A ) $.
	We claim that there is an isomorphism from the latter
	to the ring of functions of the divisor  $ \mathrm{div}( t )$,
	$$
		\gr_{\cL}(A) \longrightarrow  A/\langle t \rangle.
	$$
	This follows by Proposition \ref{Prop:Poly_behavior}(1) which implies 
	$$
	f_{\ell,*} \equiv \ini_\cL (f_\ell)(u_*, y_*) \mod t, \ \ \  \ \ \mbox{for all } \ \ell \in \{ 1, \ldots, m \},
	$$
	where $ \ini_\cL(f_\ell) = \ini_\cL (f_\ell)_{u,y} $ (Definition \ref{Def:in_0_and_u_std}(2)) and $ ( u_* ) $ are the strict transforms of $ (u) $ in $ R_* $. 
\end{Obs}

Putting everything together, we obtain the main theorem of this article.

\begin{Thm}[Theorem \ref{Thm:Intro}]
	Let $ A $ be a local Noetherian ring.
	Let $ R $ and $ \cR $ be regular local rings and $ J \subset R $ and $ \cJ \subset \cR$ be non-zero ideals such that $ R/ J \cong \cR/\cJ \cong A $.
	Let $ \la_1, \ldots, \la_d \in \IQ_+ $ be positive rational numbers and denote by $ \Lambda : \IR^d \to \IR $ the corresponding positive linear form.
	Let $ (v) = (v_1 ,\ldots, v_d) $ and $ (\cv) = (\cv_1 ,\ldots, \cv_d) $ be elements in $ A $ whose flags in $ A $ induced by $ \Lambda $ coincide,
	$$
	F^\Lambda_{\bullet, v} = F^\Lambda_{\bullet, \indcv} 
	$$
	Let $ (u) = (u_1, \ldots, u_d ) $, resp.~$ (\cu) = (\cu_1, \ldots, \cu_d ) $, be elements in $ R $, resp.~$ \cR $, mapping to $ ( v ) $, resp.~$ (\cv) $, under the corresponding canonical projection. 
	Then 
	$$
	 \delta_\Lambda (J;u) = \delta_\Lambda (\cJ;\cu)
	$$
	and, if $ \delta_\Lambda (J;u) < \infty $, then 
	$$
	 \gr_\Lambda(R/J) \cong \gr_\Lambda(\cR/\cJ).
	$$
	More precisely: the isomorphism $ R/ J \cong \cR/\cJ $ respects the filtration defined by $\Lambda$ (Definition~\ref{Def:vLambda}).
\end{Thm}

Hence, we can also write $ \delta_\Lambda(A, {v}) := \delta_\Lambda ( \cpoly Ju ) $.
Note that $ \delta_\Lambda ( A, {v} ) = \infty $ if and only if $ \cpoly Ju = \varnothing $.

\begin{proof}
	The flag $ F^\Lambda_{\bullet, v}  $ defines a permissible sequence of blowing-ups for $ X = \Spec(A) $, \eqref{eq:first_bu_seq} and \eqref{eq:second_bu_seq}.
	By Proposition \ref{Prop:Same_sequences}, the latter coincides with $ (\cS_*) $ which was constructed in the embedded situation $ J \subset R $. 
	Furthermore, $ \delta_\Lambda (J ; u) $ can be recovered from the length of $ (\cS_*)  $,
	see Theorem \ref{Thm:delta_from_S_*}.
	
	Since $ F^\Lambda_{\bullet, v} = F^\Lambda_{\bullet, \indcv}  $, they provide the same sequence of blowing-ups and in particular, this provides $ \delta_\Lambda (J;u) = \delta_\Lambda (\cJ;\cu) $.
	When $\delta_\Lambda<\infty$, the sequence \eqref{eq:second_bu_seq} of Construction~\ref{Cons:S*} is finite. 
	{Without loss of generality, we assume $ \la_1 \leq \ldots \leq \la_d $ and $ \la_d < \delta_\Lambda $.}
	Take any $g\in R$, noting that when
	 $ {N\la_1, \ldots, N \la_d, N\delta_\Lambda \in \IZ_{\geq 0}} $, the sequence 
	 $ (\cS_*) $ is the {sequence
	\eqref{eq:seq_bu_for_ref}
	 for} the system $(s)=(u,y)$ and the linear form
		 $$
		 L'(x_1,\ldots,x_{d},x_{{d}+1},\ldots,x_{d+r})
		 =
		 N \cdot {\Big(} 
		 \Lambda(x_1,\ldots,x_d) + \delta_\Lambda \cdot (x_{d+1}+\ldots+x_{d+r})
		 {\Big)}.
		 $$
		 Theorem~\ref{Thm:A4_CosRevista}(3) gives 
		 $g = t^{\mathcal N} \, \widetilde{g} \in   R(l+1,c) ${,} 
		 where 
		 {$ \mathcal N := v_{L',s}(g) = N \cdot  \delta_\Lambda \cdot \nu_\Lambda (g) $ 
		 	(Recall the definition of $ \nu_\Lambda $ in \eqref{eq:def_nu_Lamda})
		 	and}
		 $\widetilde{g}$ is the strict transform of $g$,
		 {and} $R(l+1,c)$ is the last ring of the sequence 
		 ${(\mathcal{S}_*)}$.  So for $\overline{g}\in A$, we get 
		 $\overline{g}=t^{N {\delta_\Lambda}\nu_\Lambda(\overline{g})}\overline{g}'\in A(l+1,c)${,} 
		 where 
		 $A(l+1,c) $ is the function ring of $X(l+1,c)$: $N {\delta_\Lambda}\nu_\Lambda(\overline{g})$ is the order of $\overline{g} $ along the irreducible divisor div$(t)$ of $X(l+1,c)$.
\end{proof}

\bigskip 

\section{Corollaries}

We discuss some of the consequences of Theorem \ref{Thm:Intro}.
Let $ (A, \fM, k) $ be a local Noetherian ring and $ ( v ) = (v_1, \ldots, v_d) $ be a $ A $-regular sequence such that the ring of the directrix of $ \Spec( A/ \langle v \rangle ) $ at the origin coincides with the residue field  $ k $.
Let $ R $ be a regular local ring and $ J \subset R $ be an ideal such that $ R/J \cong A $,
and $ ( u ,y ) $ be a regular system of parameters for $ R $ such that $ u_i $ maps to $ v_i $ under the canonical projection from $ R $ to $ A $.

The {\em first face} of the polyhedron $ \cpoly Ju $ is defined as the face determined by the linear form $ \Lambda_0 : \IR^d \to \IR $, $ \Lambda_0 (\x_1, \ldots, \x_d) =  \x_1 + \ldots + \x_d $.
We set
$$
	 \delta (J;u):= \delta_{\Lambda_0} (J;u).
$$

\begin{Cor}
	The number $ \delta( J;u ) $ is an invariant of the singularity $ \Spec(A) $ and the {subscheme $ \mathcal{V}:= \Spec ( R/ \langle v_1, \ldots, v_d\rangle ) $}.
	In particular, it does not depend on the choice of the systems $ ( u_1, \ldots, u_d ) $ or $ (v_1, \ldots, v_d ) $.
\end{Cor}

\begin{proof}
	The linear form $ \Lambda_0 $ is defined by $ \la_1 = \ldots = \la_d = 1 $.
	Hence the barycentric decomposition of $ \Lambda_0 $, we have $ b_1 = 1 $ and $ l = 1 $ and $ \al(1) = 1 $,
	$ \Lambda_0 (\x) = b_1 (\x_1 + \ldots + \x_d) $,
	see Observation \ref{Obs:bary} for the notations.
	But this implies that the corresponding flag $ F_\bullet^{\Lambda_0} $ (Definition \ref{Def:Flag_for_Lambda}) is given by only $ F_1 = {\mathcal V} $.
	Theorem \ref{Thm:Intro} provides the assertion. 
\end{proof}

Indeed, the sequence of blowing-ups resulting for this particular flag coincides with the one given in the proof of \cite{InvDim2} Theorem 3.15, where the above corollary is proven for $ \dim (R/J) \leq 2 $.

\begin{Cor}
	Let us fix $ A $ and $ ( v ) $.
	Let $ \cR $ be another regular local ring and $ \cJ \subset \cR$ be another non-zero ideals such that $ \cR/\cJ \cong A $, and $ (\cu) = (\cu_1, \ldots, \cu_d) $ be a system of elements in $ \cR$ that is mapped to $ (v) $ under the canonical projection.
	
	The characteristic polyhedra $ \cpoly Ju $ and $ \cpoly{\cJ}{\cu} $ coincide,
	$$
		\cpoly Ju = \cpoly{\cJ}{\cu}.
	$$
	In particular, they have the same compact faces. 
\end{Cor}

\begin{proof}
	Suppose $ \cpoly Ju \neq \cpoly{\cJ}{\cu} $.
	Then, without loss of generality, there exists a vertex $ \x \in \cpoly Ju $ which is not contained in the other polyhedron,
	$ \x \notin \cpoly{\cJ}{\cu} $.
	This implies that there is a positive linear form $ \Lambda : \IR^d \to \IR$ such that
	$ 
		\Lambda( \x' ) > \Lambda(\x),\mbox{ for every } \x' \in \cpoly{\cJ}{\cu}.
	$
	But this contradicts $ \delta_\Lambda (J;u) = \delta_\Lambda (\cJ, \cu) $, which holds by Theorem \ref{Thm:Intro}.
\end{proof}

This leads to

\begin{Def}
	We define the {\em characteristic polyhedron of the singularity $ \Spec(A) $ \wrt $ (v) $} as
	$$
	\cpoly Av := \cpoly Ju.
	$$
\end{Def}

\noindent 
Note: A natural choice for $ (v) $ is such that the ring of the directrix of $ A $ at $ \fM $ is 
	$
		k[V_1, \ldots V_d].
	$
	
\medskip	

But it may happen that one does not necessarily want to fix the entire system $ ( v ) = (v_1, \ldots, v_d ) $, but only some of them, say $ (v_I) := ( v_i \mid i \in I ) $, for $ I \subset \{ 1, \ldots, d \} $.
For example, such a situation arises along a process of resolving singularities via a sequence of blowing-ups. 
There we have the additional data of the exceptional divisors of the preceding blowing-ups which need to be taken into considerations.
Hence it is natural to fix those local coordinates defining the exceptional divisors. 
Using our main result, we can detect which part of the polyhedron is an invariant of the singularity $ \Spec(A) $ and the additional data of the exceptional divisor.
	More precisely, 
	if the flag $ F_\bullet^\Lambda $ defined by a positive linear form $ \Lambda : \IR^d \to \IR $
	is invariant under every change in $ ( v) $ that fixes the elements of $ ( v_I ) $, 
	then $ \delta_\Lambda (A;v) $ is an invariant of $ A $ and $ ( v_I ) $.

Let us illustrate this in the case that $ d = 2 $ and $ s = 1 $, i.e., we have $ ( v_1, v_2 ) $ and we fix $ v_1 $. 
Recall that the first face of the polyhedron is the one defined by the linear form $ \Lambda_0 (\x) = |\x| $.
The previous result implies that the part of $ \cpoly Av $ that is left of the first face is an invariant of $ ( A, v_1 ) $.

\smallskip 

\begin{center}
	\begin{tikzpicture}[scale=0.7]
	
	\draw (1,4)  coordinate (v1);
	\draw (1.2,3)  coordinate (v1+);
	\draw (1.8,1.8)  coordinate (v2);
	\draw (3,0.6)  coordinate (v3); 
	\draw (5.7,0.2)  coordinate (v4);

	\path[fill=softgray] (1, 5) -- (v1)  -- (v1+) -- (v2) -- (v3) -- (v4) -- (7,0.2) -- (7,5) -- (0.5,5);
	
	\draw[->] (0,0) -- (7,0) node[right]{$e_1$};
	\draw[->] (0,0) -- (0,5) node[above]{$e_2$};
	
	\draw[thick] (1, 5) -- (v1)  -- (v1+) -- (v2) -- (v3) -- (v4) -- (7,0.2);
	\draw[ultra thick] (1, 5) -- (v1)  -- (v1+) -- (v2);
	\draw[ thick,white] (1, 5) -- (v1)  -- (v1+) -- (v2);
	
	\draw[thick, dashed] (v1) -- (1,-0.1);
	\draw[thick, dashed] (v1) -- (-0.1, 4);
	\draw[thick, dashed] (v2) -- (-0.1, 1.8);

	\node[left] at (-0.1,4) {$ \be $};
	\node[left] at (-0.1,1.8) {$ \ga $};
	\node[below] at (1,-0.1) {$ \al $};
	
	\draw[fill = white] (v1) circle [radius=0.085];
	\draw[fill = white] (v1+) circle [radius=0.085];
	\draw[fill = white] (v2) circle [radius=0.085];
	\draw[fill = black] (v3) circle [radius=0.08];
	\draw[fill = black] (v4) circle [radius=0.08];

	\end{tikzpicture}  
	
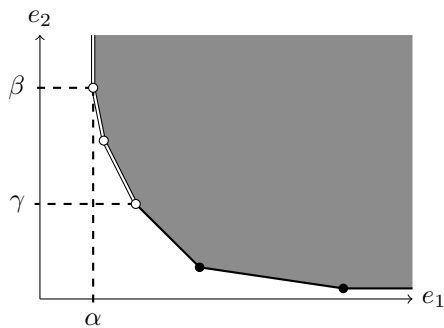
	{\bf Figure 5:} Illustration of $ \al, \be, \ga $;
	the left part of the polyhedron (white) is an invariant of $ A = R/J $ and $ (v_1) $, while the remaining part (black) is not.
\end{center}

\smallskip 

\noindent 
In particular, the following numerical data are invariants of $ A $ and $ v_1 $:
\begin{itemize}
	\item $ \al := \al(A,v) := \inf\{ \x_1 \in \IR_{\geq 0} \mid (\x_1, \x_2 ) \in \cpoly Av \} $. 
	\item $ \be := \be(A,v) := \inf\{ \x_2 \in \IR_{\geq 0} \mid (\al,\x_2) \in \cpoly Av )\} $.
	\item $ \ga := \ga (A,v) := \inf \{ {\x_2} \in \IR_{\geq 0} \mid {(\delta(A,v) -\x_2, \x_2 )} \in \cpoly Av \} $. 
\end{itemize}
We also refer to \cite{InvDim2} Theorem 3.18, where the latter was proven directly and then theses numbers where used in the construction of an invariant measuring the improvement of a singularity along a given resolution process. 
Hence, the understanding of the part of the characteristic polyhedron that is actually an invariant of the singularity $ \Spec(A) $ and some divisors corresponding to fixed elements $ ( v_i )_{i\in I} $ 
may provide new insights when looking for invariants for resolution of singularities in dimension three and larger.

\setcounter{section}{0}
\renewcommand{\thesection}{\Alph{section}}

\bigskip 

\section{Appendix: Computing the Directrix from the Ridge}\label{app:calculdeladirectrice}

Let $ k $ be a field and $ C = V ( I ) \subset \A^n_k $ be a cone defined by some homogeneous ideal $ I \subset k [X_1, \ldots, X_n ] $. 

The directrix of $ C $ is a very natural notion that has important applications in the study of singularities and their resolution, in particular,
one likes to find a minimal set of variables $ (Y_1, \ldots, Y_r ) $ such that there exists a system of generators for $ I $ contained in $ k [Y_1, \ldots, Y_r ] $, i.e., to find a minimal vector space $V\subset<X_1, \ldots, X_n >_k$   such that there exists a system of generators for $ I $ contained in $k[V]$. 
Formally, the {\em directrix of $ C $} (\cite{CJS}~Definitions 1.8 and 1.26) is defined as the biggest $ k $-subvector space $ W $ of $ \A^n_k = \Spec (k [X_1, \ldots, X_n])$ leaving $ C $ stable under translation, i.e., for which we have
$$
C + W  = C.
$$
This vector space has for equations $(Y_1, \ldots, Y_r ) \subset k [X_1, \ldots, X_n ]$ and $ k [Y_1, \ldots, Y_r ]=k[V] $ is its algebra of invariants as a group of translations of $ \A^n_k $.

A generalization for this is the {\em ridge}.  It is the  group $F$ of translations of $\A^n_k $  which leave stable $C$.  See the precise definition in~\cite{GiraudMaxPos}~1.5. Its algebra of invariants $k[U]$ is the smallest algebra of  $ k [X_1, \ldots, X_n] $ generated by homogeneous additive polynomials $(\sigma_1,\ldots,\sigma_\tau)$  such that there exists a system of generators for $ I $ contained in $k[U]=k[\sigma_1,\ldots,\sigma_\tau]$ (the algebra of invariants of $F$). When char$(k)=0$, ridge and directrix coincide.

By definitions, the directrix is a subgroup of the ridge,  so $k[U]\subset k [V ] $. One can compute the ridge using differential operators, \cite{GiraudMaxPos}~Lemma~1.7. In \cite{ComputeRidge}, there is an effective algorithm to compute the ridge. 
Assume that $(F_1,\ldots,F_m)$ is an homogeneous standard basis of $ I \subset k [X_1, \ldots, X_n ] $. Then, by  \cite{GiraudMaxPos}~Lemma~1.7 and \cite{ComputeRidge}, $k[U]$ is generated by the $D_A^{(X)} F_i$ with $\vert A \vert<n_i:= \deg(F_i)$, where $A=(a_1,\ldots,a_n)$ is a multi index and $D_A^{(X)}$ is the Hasse-Schmidt derivation. Note that 
$$D_A^{(X)} F_i \in k[U], \ \vert A \vert<n_i.$$
 Furthermore, up to a change of ordering on the $X_i$, there is a minimal set of generators:

$$	
	\begin{array}{lcll}
		\si_1= X_1^{q_1}+\sum_{i\geq 2} \lambda_{1,i} \, X_i^{q_1},
		\\[10pt] 
		\si_2= X_2^{q_2}+\sum_{i\geq 3} \lambda_{2,i} \, X_i^{q_2},
		\\
		\hspace{16pt} \vdots
		\\[3pt] 
		\si_\tau= X_\tau^{q_\tau}+\sum_{i\geq 1+\tau} \lambda_{\tau,i} \, X_i^{q_\tau}.
	\end{array}
$$
$$	\begin{array}{c}
		\mbox{where } q_{j-1}\leq q_{j},
		\ \lambda_{j,i} \in k,
		\mbox{ for } 1\leq j < i \leq n,
		\\
		q_j =1 \ \hbox{when char}(k)=0, \ q_j \ \hbox{is a }p\hbox{-power when char}(k)=p >0.
	\end{array}
$$
In the case char$(k)=p>0$, let us construct $ k [V] $. Take any additive polynomial $P\in k[U]$.
$$
	P=\sum_{1\leq i \leq n}\lambda_i X_i^q,\ \lambda_i \in k,\ q=p^d,\ d\in \IZ_{\geq 0}.
$$
When $d=0$, as $k[U]\subset k [V ] $, then $P\in V$. When $d\geq 1$, take a basis $(\mu_1,\ldots,\mu_l)$ of $k^q[\lambda_1,\ldots,\lambda_n]$ over $k^q$ (a $q$-basis of $k^q[\lambda_1,\ldots,\lambda_n]$).  Then 
$$P=\sum_{1\leq j \leq l} \mu_j L_j^q,$$
with $L_j$ linear form in $(X_1,\ldots,X_n)$. To compute  $V$, one has to take a set of additive polynomials  generating $k[U]$ as a $k$-algebra, for example, take all the $\si_j,\ 1\leq j \leq \tau$, make the construction above, all the linear forms $L_j$ computed generate $V$ as a $k$-vector space.

\begin{Rk}
	\label{Rk:Pasdecidable}
	Fr\"{o}hlich and Shepherdson \cite{Pasdecidable}~Section~7 
have  shown that testing if an element is a $p$-th power is not decidable in
general. So our method to compute the directrix is far from being efficient in general. 
\end{Rk}


\begin{thebibliography}{CPSc}
	
	
	
	 
	\bibitem[BHM]{ComputeRidge}
	J.~Berthomieu, P.~Hivert, and H.~Mourtada,
	\newblock {Computing {H}ironaka's invariants: ridge and directrix},
	\newblock in {\em Arithmetic, geometry, cryptography and coding theory 2009}, 9--20,
	{\em Contemp. Math.} {\bf 521}, Amer. Math. Soc. (2010).
	
	
	
	\bibitem[C]{CosRevista}
	V.~Cossart,
	\newblock {Poly\`{e}dre caract\'{e}ristique et \'{e}clatements
		combinatoires},
	\newblock {{\em Rev. Mat. Iberoam.} {\bf 5} (1989), 67--95.}
	
	
	\bibitem[CGO]{CGO}
	V.~Cossart, J.~Giraud, and U.~Orbanz,
	\newblock {Resolution of surface singularities},
	\newblock with an appendix by H.~Hironaka,
	\newblock {\em Lecture Notes in Mathematics} {\bf 1101}, Springer Verlag, (1984).
	
	
	
	\bibitem[CJS]{CJS}
	V.~Cossart, U.~Jannsen, and S.~Saito,
	\newblock {Canonical embedded and non-embedded resolution of singularities for
		excellent two-dimensional schemes},
	\newblock {\em preprint} (2009), available on arXiv:0905.2191.
	
	
	\bibitem[CP1]{CPcompl}	
	V.~Cossart and O.~Piltant,
	\newblock {Characteristic polyhedra of singularities without completion},
	\newblock {{\em Math. Ann.} {\bf 361} (2015), 157--167, 
		DOI 10.1007/s00208-014-1064-0.}
	
	
	\bibitem[CP2]{CPmixed}
	V.~Cossart and O.~Piltant,
	\newblock {Resolution of Singularities of  Threefolds in mixed characteristic: Case of small multiplicity},
	\newblock {\em Revista de la Real Academia de Ciencias Exactas, Fisicas y Naturales. Serie A. Matematicas} {\bf 108} (2014), 113--151.
	
	
	\bibitem[CP3]{CPmixed2}
	V.~Cossart and O.~Piltant,
	\newblock {Resolution of Singularities of Arithmetical Threefolds {II}},
	\newblock {\em preprint} (2014), available on {HAL:hal01089140}, 
	to be published in Jour. of Algebra.
	
	
	\bibitem[CPSc]{CPS}
	V.~Cossart, O.~Piltant and B.~Schober,
	\newblock {Fa\^ite du c\^one tangent \`a une singularit\'e: un th\'eor\`eme oubli\'e}, {\em C.~R.~Acad.~Sci. Paris, Ser. I} {\bf 355} (2017), 455--459. 
	
	\bibitem[CSc1]{InvDim2}
	V.~Cossart and B.~Schober,
	\newblock {A strictly decreasing invariant for resolution of singularities in dimension two},
	\newblock {\em preprint} (2014), available on arXiv:1411.4452.
	
		
	\bibitem[CSc2]{CSCcompl}
	V.~Cossart and B.~Schober,
	\newblock {Characteristic polyhedra of singularities without completion - Part II},
	\newblock {\em preprint} (2014), available on arXiv:1411.2522.
	
	
		
	\bibitem[FS]{Pasdecidable}
	A.~Fr\"{o}hlich  and J.C.~Shepherdson, 
	\newblock{Effective procedures in field theory}, 
	\newblock {\em Philos. Trans.
Roy. Soc. London.} Ser. A., {\bf 248} (1956), 407--432.
	

	\bibitem[G1]{GiraudEtude}
	J.~Giraud,
	\newblock {\'{E}tude locale des singularit\'{e}s},
	\newblock {Cours de $3^{\mbox{\`{e}me}}$ cycle, 1971-1972, {\em Publ. Math. d'Orsay} {\bf 26} (1972)}.
	
	\bibitem[G2]{GiraudCharZero}	
	J.~Giraud,
	\newblock {Sur la th\'{e}orie du contact maximal}.
	\newblock {\em Math. Z.} {\bf 137} (1974), 285--310.
	
	
	\bibitem[G3]{GiraudMaxPos}
	J.~Giraud,
	\newblock {Contact maximal en caract\'{e}ristique positive},
	\newblock {\em Ann. Sci. \'Ec. Norm. Sup.} $4^{\mbox{{\`e}me}}$ s\'erie { \bf 8} (1975), 201--234.
	
	
	
	\bibitem[H1]{HiroCharPoly}
	H.~Hironaka,
	\newblock {Characteristic polyhedra of singularities},
	\newblock {\em J. Math. Kyoto Univ.} {\bf 7} (1967), 251--293.
	
	
	\bibitem[H2]{HiroBowdoin}
	H.~Hironaka,
	\newblock {Desingularization of Excellent Surfaces},
	\newblock Adv. Sci. Sem. in Alg. Geo, Bowdoin College, Summer 1967.
	\newblock Notes by Bruce Bennett, appendix of \cite{CGO} (1984).
	
	
	
	\bibitem[H3]{HiroMero}
	H.~Hironaka,
	\newblock {Bimeromorphic smoothing of a complex analytic space},
	\newblock {\em Acta Math. Vietnam.} {\bf 2} (1977), no. 2, 103--168. 
	
	
	\bibitem[Sc1]{SCarthIdExp}
	B.~Schober,
	\newblock {Idealistic exponents and their characteristic polyhedra},
	\newblock {\em preprint} (2014), available on arXiv:1410.6541.
	
	
	\bibitem[Sc2]{SCBM}
	B.~Schober,
	\newblock {A polyhedral approach to the invariant of Bierstone and Milman},
	\newblock {\em preprint} (2014), available on arXiv:1410.6543.
	
\end{thebibliography}
\end{document}